\documentclass[reqno,english]{amsart}
\usepackage{amsfonts,amsmath,latexsym,verbatim,amscd,mathrsfs,color,array}

\usepackage{amsmath,amssymb,amsthm,amsfonts,graphicx,color}
\usepackage{amssymb}
\usepackage{pdfsync}
\usepackage{epstopdf}

\usepackage[english]{babel}
\usepackage{amscd}
\usepackage{verbatim}
\usepackage{float}

\newcommand{\R}{\mathbb{R}}

\newcommand{\be}{\begin{equation}}
\newcommand{\ee}{\end{equation}}
\newcommand{\bea}{\begin{eqnarray}}
\newcommand{\eea}{\end{eqnarray}}
\newcommand{\bean}{\begin{eqnarray*}}
\newcommand{\eean}{\end{eqnarray*}}

\newcommand{\xa}{\alpha}
\newcommand{\xb}{\beta}
\newcommand{\xg}{\gamma}
\newcommand{\xG}{\Gamma}
\newcommand{\xd}{\delta}
\newcommand{\xD}{\Delta}
\newcommand{\xe}{\varepsilon}
\newcommand{\xz}{\zeta}

\newcommand{\xl}{\lambda}
\newcommand{\xL}{\Lambda}

\newcommand{\xm}{\mu}
\newcommand{\xn}{\nu}

\newcommand{\xs}{\sigma}

\newcommand{\xf}{\phi}
\newcommand{\xF}{\Phi}

\newcommand{\xo}{\omega}

\newcommand{\inn}{{\quad\hbox{in } }}

\newcommand{\foral}{\quad\mbox{for all}\quad}

\newtheorem{theorem}{Theorem}[section]
\newtheorem{lemma}[theorem]{Lemma}
\newtheorem{proposition}[theorem]{Proposition}

\newtheorem{notation}[theorem]{Notation}

\begin{document}

\title[ Ancient  solutions to the Allen-Cahn equation]
{ Ancient multiple-layer solutions to the Allen-Cahn equation}


\author[M. del Pino]{Manuel del Pino}
\address{\noindent   Departamento de
Ingenier\'{\i}a  Matem\'atica and Centro de Modelamiento
 Matem\'atico (UMI 2807 CNRS), Universidad de Chile,
Casilla 170 Correo 3, Santiago,
Chile.}
\email{delpino@dim.uchile.cl}

\author[K. T. Gkikas]{Konstantinos T. Gkikas}
\address{\noindent  Centro de Modelamiento
 Matem\'atico (UMI 2807 CNRS), Universidad de Chile,
Casilla 170 Correo 3, Santiago,
Chile.}
\email{kgkikas@dim.uchile.cl}

\date{}\maketitle


\begin{abstract}
We consider the parabolic one-dimensional Allen-Cahn equation
$$u_t= u_{xx}+ u(1-u^2)\quad (x,t)\in \R\times (-\infty, 0].$$ The steady state
$w(x) =\tanh (x/\sqrt{2})$, connects, as a ``transition layer'' the stable phases $-1$ and $+1$.
We construct a solution $u$  with any given number $k$ of transition layers between $-1$ and $+1$. At main order they consist of $k$ time-traveling copies of $w$ with interfaces diverging  one to each other as $t\to -\infty$.
More precisely, we find
$$
u(x,t) \approx   \sum_{j=1}^k (-1)^{j-1}w(x-\xi_j(t))  + \frac 12 ( (-1)^{k-1}-  1)\quad \hbox{ as } t\to -\infty,
$$
where  the functions $\xi_j(t)$ satisfy a first order Toda-type system. They are given by
$$\xi_j(t)=\frac{1}{\sqrt{2}}\left(j-\frac{k+1}{2}\right)\log(-t)+\xg_{jk},\quad j=1,...,k,$$
for certain explicit constants $\xg_{jk}.$


\end{abstract}









\date{}\maketitle

\setcounter{equation}{0}
\section{Introduction and statement of the main result}

A classical model for phase transitions is the Allen-Cahn equation \cite{ac}
\be u_t = \Delta u + f(u) \inn \R^N,   \label{ac0} \ee
where $f(u) = - F'(u)$ where $F$ is a {\em balanced bi-stable potential}
namely $F$ has exactly two non-degenerate global minimum points $u=+1$ and $u=-1$. The model is
$$F(u) = - \frac 14(1-u^2)^2,$$
 so that $f(u) = (1-u^2)u $.
The constant functions  $u=\pm 1$ correspond to stable equilibria of Equation \eqref{ac0}.  They are idealized as two phases of a
material. A solution $u(x)$ whose values lie at all times in $[-1,1]$ and in most of the space $\R^N$ takes values close to either $+1$ or $-1$ corresponds to a continuous realization of the phase state of the material, in which the two stable states coexist.
There exists a large literature on this type of solutions (in the static and dynamic cases). The main point is to derive qualitative
information on the ``interface region'', that is the walls separating the two phases. A close connection between these walls and minimal surfaces and surfaces evolving by mean curvature has been established in many works. We refer the reader for instance to
\cite{5authors,dkwcpam,hamel0,hamel2,matano,ninomiya,taniguchi}.
On the other hand, the main difference between interfaces and surfaces evolving mean curvature surfaces, is that in the phase transition model different components do interact giving rise to interesting motion patterns.

\medskip
The purpose of this paper is to study multiple-interface interaction in the simplest, one-dimensional scenario. We will construct non stationary solutions defined at all times, which in the {\em ancient regime} multiple, quite separated transitions are present, with a dynamical law that is rigorously established. More precisely, we consider the problem of building {\em ancient solutions} $u(x,t)$ to the
one-dimensional Allen-Cahn equation \cite{ac}

\be
u_t= u_{xx} +u(1-u^2) \inn \R\times (-\infty ,0] \label{alen},
\ee
which exhibit a finite number of transitions that connect the values $-1$ and $+1$.

\medskip
The building blocks of these solutions are the single-transition layer equilibrium solutions to   \eqref{alen}
$$ u''+u(1-u^2)=0\quad\hbox{in}\;\mathbb{R},\quad\lim_{x\rightarrow\infty}u(x)=1,\quad\quad\lim_{x\rightarrow-\infty}u(x)=-1,$$
which in phase plane represents a heteroclinic monotone connection between the constant equilibria $\pm 1$.
This solution is unique up to translations. The unique one with $u(0)=0$ will be denoted from now on  $w(x)$ and it is
 given in closed form by
 \be w(x)=\tanh\left(\frac{x}{\sqrt{2}}\right ). \label{w}\ee

Given an even number $k$,  we want to build a solution $u(x,t)$ to \eqref{alen} that near each of $k$ ordered, very distant ``transition points''
$\xi_j(t),\  j=1,\ldots, k$
satisfies
$$u(x,t) \approx \pm w(x-\xi_j(t)). $$
More precisely, we want to find a solution of the form
\be
u(t,x)=   -1+  \sum_{j=1}^{k}(-1)^{j+1}w(x-\xi_j(t))  +\psi(t,x),\label{formin}
\ee
with
\be
\xi_1(t)<\xi_2(t)<...<\xi_k(t),\qquad\xi_j(t)=-\xi_{k-j+1}(t),\label{formin2}
\ee
where the perturbation function  $\psi(t,x)$  goes to zero uniformly as
$t\rightarrow-\infty$ and
 satisfies the orthogonality conditions
\be
\int_{\mathbb{R}}\psi(t,x)w'(x-\xi_i(t))dx=0\foral i=1,...,k,\;t<-T, \label{orthcondin}
\ee
for a suitable large $T>0$.
We shall establish the existence of a solution with this characteristic. In fact, as we will see the interface
dynamic is driven at main order by the following system of differential equations (a first order Toda system)
\be
\frac{1}{\xb}\xi_j'-e^{-\sqrt{2}(\xi_{j+1}-\xi_j)}+e^{-\sqrt{2}(\xi_{j}-\xi_{j-1})}=0,\quad j=1,...,k,\;\quad\;t\in(-\infty, 0].\label{toda1}
\ee

The dynamic law of interface interaction was formally derived in a related Neumann problem by Fusco and Hale \cite{fuscohale}, see also \cite{cp1,cp2}.
In \cite{chen} Chen, Guo and Ninomiya built a solution with two transition layers traveling in opposite directions (the case $k=2$ for us). The
argument employed there was based on barriers, and it is not clear to us how to extend it to multiple transitions.
In \cite{dds} the first order Toda system appears in the construction of ancient solutions for the Yamabe flow.

\medskip
 More precisely, we will find
$$\xi_j(t)=\xi^0_j(t)+h_j(t),\quad j=1,...,k,$$
for some suitable parameter functions $h_j(t),$ such that the parameter functions $h_j(t)$ will decay in $|t|,$ as $t\rightarrow-\infty$ for all $j=1,...,k$
and the functions $\xi^0_j$  solve the first order Toda system,
\be
\frac{1}{\xb}\xi_j'-e^{-\sqrt{2}(\xi_{j+1}-\xi_j)}+e^{-\sqrt{2}(\xi_{j}-\xi_{j-1})}=0,\quad j=1,...,k,\;\; 
\quad\;t\in(-\infty , 0],\label{toda}
\ee
with the conventions
$$\xi_{k+1}=\infty\quad \hbox{and}\quad \xi_0=-\infty,$$
where  $T_0>0$ and
\be
\xb=\frac{6\int_{\mathbb{R}}e^\frac{2x}{\sqrt{2}}(1-w^2(x))w'(x)dx}{\int_{\mathbb{R}}(w'(x))^2dx}.\label{beta}
\ee
We will see that a solution of the above system is given by
\be\xi_j^0(t)=\frac{1}{\sqrt{2}}\left(j-\frac{k+1}{2}\right)\log(-2\sqrt{2}\xb t)+\xg_{jk},\quad j=1,...,k,\;\; 
\label{pp}\ee
for certain explicit constants $\xg_{jk}.$

Our main result states as follows.
\begin{theorem}
Let $k\geq2$ be an even integer and  $\xi^0_j$ be the solution \eqref{pp} of the Toda system \eqref{toda}. Then there exists a number $T >0$ and a solution $u(t,x)$ to (\ref{alen}) defined on $(-\infty,-T ]\times\mathbb{R},$
of the form (\ref{formin})-(\ref{orthcondin}), with $\xi$ of the form:
$$\xi_j(t)=\xi^0_j(t)+h_j(t),\quad j=1,...,k.$$
where the functions  $\psi(t,x),$ and $h_j(t)$ tend to zero in suitable uniform
norms as $t\rightarrow-\infty.$\label{maintheorem}
\end{theorem}

If $k$ is odd a similar construction can be made, with slightly different asymptotic configurations. For notational simplicity we will only consider the case of an even $k$ in this paper.

\setcounter{equation}{0}
\section{The first approximation}

We want to solve the problem,
$$u_t= u_{xx}+f(u),\qquad\mathrm{in}\;\;(-\infty,-T)\times\mathbb{R},$$
where $f(u)=u(1-u^2),$ and $T$ is a large positive number whose value can be adjusted at different steps.

Let $k\geq2$ be an even integer. We set
$$w_i(t,x)=w(x-\xi_i(t)),$$
where the functions $\xi_i(t)$ are ordered and symmetric,
$$\xi_1(t)<\xi_2(t)<...<\xi_{\frac{k}{2}}(t)<0<\xi_{\frac{k}{2}+1}<...<\xi_k(t),\qquad\xi_j (t) =-\xi_{k-j+1}(t).$$
We set $\xi(t)=(\xi_1(t),....,\xi_k(t))^T.$ And we write
\be
\xi(t):=\xi^0(t)+h(t),\label{ksij}
\ee
where
$$\xi^0_j=\frac{1}{\sqrt{2}}\left(j-\frac{k+1}{2}\right)\log(-2\sqrt{2}\xb t)+\xg_j,$$
$\xb$ has been defined in \eqref{beta} and $\xg_j$ are constants which we will determine them later. In addition, the function $h(t)$ satisfies
$||h(t)||_{L^\infty}+||(t-1)h'(t)||_{L^\infty}\leq1$ and $\lim_{t\rightarrow-\infty}|h(t)|+|h'(t)|=0.$

\medskip
We look for a solution of the form
\be
u(t,x)=\sum_{j=1}^{k}(-1)^{j+1}w_j(t,x)-1+\psi(t,x).\label{form}
\ee
Set
\be
z(t,x)=\sum_{j=1}^{k}(-1)^{j+1}w_j(t,x)-1\label{z}.
\ee
We would like $\psi$ to satisfy
\begin{align}
\psi_t=\psi_{xx}+f'(z(t,x))\psi+E+N(\psi)&-\sum_{i=1}^k c_i(t)w'(x-\xi_i(t)), \quad (-\infty,-T)\times\mathbb{R}, \label{mainpro}\\
\int_{\mathbb{R}}\psi(t,x)w'(x-\xi_i(t))dx&=0,\qquad\forall i=1,...,k,\;t<-T, \label{orthcond}
\end{align}
where
\be
E=\sum_{j=1}^{k}(-1)^{j+1}w'(x-\xi_j(t))\xi_j'(t)+f(z(t,x))-\sum_{j=1}^{k}(-1)^{j+1}f(w_j(t,x)),\label{error}\ee
$$N(\psi)=f(\psi(t,x)+z(t,x))-f(z(t,x))-f'(z(t,x))\psi,$$
and $c_i(t)$ have been chosen such that $\psi$ satisfies the orthogonality condition (\ref{orthcond}), namely in such a way that the following (nearly diagonal) system holds
\begin{align}\nonumber
\sum_{i=1}^kc_i(t)\int_{\mathbb{R}}&w'(x-\xi_i(t))w'(x-\xi_j(t))dx\\ \nonumber
&=\int_{\mathbb{R}}\left(\psi_{xx}(t,x)+f'(z(t,x))\psi(t,x)\right) w'(x-\xi_j(t)) dx\\ \nonumber
&-\xi'_j(t)\int_{\mathbb{R}}\psi(t,x)w''(x-\xi_j(t))dx\\
&+\int_{\mathbb{R}}(E+N(\psi))w'(x-\xi_j(t))dx,\qquad\forall i=1,...,k,\;t<-T.
\end{align}

\medskip
\noindent Later we will choose $h(t)$ such that $c_i(t)=0,\;\foral\;i=1,...,k.$

In the rest of this work we use the following notations
\begin{notation}
i)
\be
\xi=\xi^0+h,\label{xi}
\ee
where $h:\mathbb{R}\mapsto\mathbb{R}^k$ is a function that satisfies
\be\sup_{t\leq -1}|h(t)|+\sup_{t\leq -1}|t||h'(t)|<1\label{h}.
\ee
ii)
$$z(t,x)=\sum_{j=1}^{k}(-1)^{j+1}w(x-\xi_j(t))-1.$$
\label{notation}
\end{notation}

In the following lemma we find a bound for the error term $E=E(t,x)$ in \eqref{error}.
\begin{lemma}
Let $T_0>1,$ $0<\xs<\sqrt{2},$ we define
\bea\nonumber
\xF(t,x)&=&e^{\xs(-x+\xi_{j-1}^0(t))}+e^{\xs(x-\xi_{j+1}^0(t))},\\
&&\mathrm{if}\;\;\qquad\frac{\xi_j^0(t)+\xi_{j-1}^0(t)}{2}\leq x\leq \frac{\xi_j^0(t)+\xi_{j+1}^0(t)}{2},\;j=1,...,k,\label{bound}
\eea
with $\xi_0^0=-\infty$ and $\xi_{k+1}^0=\infty.$ Then there exists a uniform constant $C>0$ which depends only on $k,$ such that
$$|E(t,x)|\leq C\xF(t,x),\quad\forall (t,x)\in(-\infty, -T_0]\times\mathbb{R},$$
where $E$ is the error term in \eqref{error} and $\xi$ satisfies the assumptions \eqref{xi} and \eqref{h}.\label{remark}
\end{lemma}
\begin{proof}

First note that there exists a positive constant $c:=c(\xg_1,...,\xg_k,\xb)>0$ such that the following inequality holds

$$\sup_{x\in\mathbb{R}}\left\{\frac{w'(x-\xi_j(t))}{\xF(t,x)}\right\}\leq c|t|^{\frac{\xs}{\sqrt{2}}},\quad\forall\;j=1,...,k.$$
Using the fact that $$|\xi_j'|\leq C_1(\xb,k) |t|^{-1},$$
we obtain that there exists a positive constant $C_2=C_2(k)$ such that
$$\sup_{x\in\mathbb{R}}\left\{\frac{w'(x-\xi_j(t))}{\xF(t,x)}\right\}|\xi_j'|\leq C_2|t|^{\frac{\xs}{\sqrt{2}}-1},\quad\forall\;j=1,...,k.$$

Now, let $$\frac{\xi_j^0(t)+\xi_{j-1}^0(t)}{2}\leq x\leq \frac{\xi_j^0(t)+\xi_{j+1}^0(t)}{2},\;j=1,...,k,$$ with $\xi_0^0=-\infty$ and $\xi_{k+1}^0=\infty.$

If $i\leq j-1,$ by our assumptions on $\xi_i,$ there exists a uniform constant $C>0$ such that

$$|w(x-\xi_i(t))-1|\leq Ce^{\sqrt{2}(-x+\xi_{j-1}^0(t))}.$$
Similarly if $i\geq j+1$

$$|w(x-\xi_i(t))+1|\leq Ce^{\sqrt{2}(x-\xi_{j+1}(t))}.$$

We set
$$g=\sum_{i=1}^{j-1}(-1)^{i+1}\left(w(x-\xi_i)-1\right)+\sum_{i=j+1}^k(-1)^{i+1}\left(w(x-\xi_i)+1\right).$$

Then
\begin{align}\nonumber
&\left|f\left(g+(-1)^{j+1}w(x-\xi_j(t))\right)-\sum_{i=1}^{k}(-1)^{i+1}f(w_i(t,x))\right|\\ \nonumber
&=\left|f(g)+(-1)^{j+1}f(w_j(t,x))-(-1)^{j+1}3g^2w_j-3gw^2_j-\sum_{i=1}^{k}(-1)^{i+1}f(w_i(t,x))\right|\\ \nonumber
&\leq C|g|+ \left|f(g)-\sum_{i=1,\;i\neq j}^{k}(-1)^{i+1}f(w_i(t,x))\right|\\ \nonumber
&\leq C\left(\sum_{i=1}^{j-1}|w(x-\xi_i)-1|+\sum_{i=j+1}^k(-1)^{i+1}|w(x-\xi_i)+1|\right).
\end{align}
Combining all above and using the properties of $\xi$ we can reach to the desired result.
\end{proof}

\setcounter{equation}{0}
\section{The linear problem}

This section is devoted to build a solution to the linear parabolic problem

\be
\psi_t=\;\psi_{xx}+f'(z(t,x))\psi+h(t,x)-\sum_{j=1}^k c_i(t)w'(x-\xi_j(t)),\qquad \mathrm{in}\;\;\;(-\infty,-T_0]\times\mathbb{R},\label{fixprop}
\ee
\begin{align}
\int_{\mathbb{R}}\psi(t,x)w'(x-\xi_i(t))dx&=0,\qquad\forall i=1,...,k,\ t\in (-\infty,-T_0],  \label{orthcond1}
\end{align}
for a bounded function $h,$ and  $T_0>0$ fixed sufficiently large.

The numbers $c_i(t)$ are exactly those that make the relations above consistent,
namely, by definition for each $t<-T_0$ they solve the linear system of equations
\begin{align}\nonumber
\sum_{i=1}^kc_i(t)\int_{\mathbb{R}}w'(x-\xi_i(t))w'(x-\xi_j(t))dx&=\int_{\mathbb{R}}\left(\psi_{xx}+f'(z)\psi\right) w'(x-\xi_j(t)) dx\\ \nonumber
&-\xi'_j(t)\int_{\mathbb{R}}\psi(t,x) w''(x-\xi_j(t))dx\\
&+\int_{\mathbb{R}} h w'(x-\xi_j(t))dx,\quad\ j=1,...,k .
\end{align}
This system can indeed be solved uniquely since if $T_0$ is taken sufficiently large, the matrix  with coefficients $\int_{\mathbb{R}}w'(x-\xi_i(t))w'(x-\xi_j(t))dx$
is nearly diagonal.

Our purpose is to build a linear operator $\psi = A(h)$ that defines a solution of \eqref{fixprop}-\eqref{orthcond1} which is bounded for norm suitably adapted to our setting.

\medskip

Let $\mathcal{C}_\xF((s,t)\times\mathbb{R})$ is the space of continuous functions
with norm
$$
||u||_{\mathcal{C}_\xF((s,t)\times\mathbb{R})}=\left|\left|\frac{u}{\xF}\right|\right|_{L^\infty((s,t)\times\mathbb{R})},$$
where $\xF$ has been defined in \eqref{bound}.

\begin{proposition} \label{prop1} There exist positive numbers $T_0$ and $C$ such that
for each $h\in\mathcal{C}_\xF((-\infty ,0)\times\mathbb{R}),$ there exists a solution
of problem \eqref{fixprop}-\eqref{orthcond1}  $\psi = A(h)$ which defines
a linear operator of $h$ and satisfies the estimate
\be
||\psi||_{\mathcal{C}_\xF((-\infty,t)\times\mathbb{R})}\leq C||h||_{\mathcal{C}_\xF((-\infty,t)\times\mathbb{R})},\qquad\forall t\leq -T_0.
\label{estfix**}
\ee\label{fixth}
\end{proposition}

The proof will be a consequence of intermediate steps that we state and prove next.
Let $g(t,x)\in \mathcal{C}_\xF\left((-\infty,-T)\times\mathbb{R}\right).$ For $T>0$ and $s<-T$ we consider the Cauchy problem
\bea\nonumber
\psi_t&=&\; \psi_{xx}+f'(z(t,x))\psi+g(t,x),\qquad \mathrm{in}\;\;\;(s,-T]\times\mathbb{R},\\
\qquad\qquad\qquad\psi(s,x)&=&0,\phantom{\psi_{xx}+f'(z)\psi+g(t,x)}\qquad \mathrm{in}\;\;\;\mathbb{R},\label{mainpro*}
\eea
which is uniquely solvable. We call $T^s(t,x)$ its solution.  By standard regularity theory we have $T^s\in C^{0,\xa}((s,-T)\times\mathbb{R}).$

\subsection{A priori estimates for the solution of the problem (\ref{mainpro*})}

 We will establish in this subsection a priori estimates for the
solutions  $T^s$ of (\ref{mainpro*}) that are independent on $s.$
\begin{lemma}

Let $T^s\in \mathcal{C}_\xF\left((s,-T)\times\mathbb{R}\right)$ be a solution of the problem (\ref{mainpro*})\\ and $g(t,x)\in \mathcal{C}_\xF\left((s,-T)\times\mathbb{R}\right),$ satisfies
\begin{align}\nonumber
&\int_{\mathbb{R}}g(t,x)w'(x-\xi_j(t))dx=-\int_{\mathbb{R}}\left( T^s_{xx}(t,x)+f'(z(t,x))T^s(t,x)\right) w'(x-\xi_j(t)) dx\\
&+\xi'_j(t)\int_{\mathbb{R}}T^s(t,x)w''(x-\xi_j(t))dx,\qquad\forall i=1,...,k,\;s<t<-T.\label{cond}
\end{align}

\hfill$\forall i=1,...,k,\;t<-T.$

Then there exists a uniform constant $T_0>0$ such that for any $t\in (s,-T_0],$ the following estimate is valid
\be
||T^s||_{\mathcal{C}_\xF((s,t)\times\mathbb{R})}\leq C||g||_{\mathcal{C}_\xF((s,t)\times\mathbb{R})},\label{estfix}
\ee
where $C>0$ is a uniform constant.\label{mainlem}
\end{lemma}
\begin{proof}
We note here that the assumption (\ref{cond}) implies
\be
\int_{\mathbb{R}}T^s(t,x)w'(x-\xi_i(t))dx=0,\qquad\forall i=1,...,k,\;s<t<-T. \label{condpsi}
\ee
Indeed since $T^s$ is the solution of (\ref{mainpro*}), using $w'(x-\xi_j(t))$ for test function, we have for any $t\in (s,-T]$
\begin{align*}
\int_{\mathbb{R}}T^s_tw'(x-\xi_j(t))dx=-\;\int_{\mathbb{R}}T_{x}^sw''(x-\xi_j(t))dx&+\int_{\mathbb{R}}f'(z(t,x))T^s w'(x-\xi_j(t))dx\\
&+\int_{\mathbb{R}}g(t,x)dx.
\end{align*}
But by \eqref{cond} we have $\forall\; t\in (s,-T]$
\begin{align*}
&\int_{\mathbb{R}}g(t,x)w'(x-\xi_j(t))dx=\int_{\mathbb{R}}T^s_x(t,x)w''(x-\xi_j(t))dx\\
&-\int_{\mathbb{R}}f'(z(t,x))T^s(t,x)w'(x-\xi_j(t)) dx+
\xi'_j(t)\int_{\mathbb{R}}T^s(t,x)w''(x-\xi_j(t))dx.
\end{align*}
Thus combining all above we have that
\begin{align*}
\int_{\mathbb{R}}T^s_tw'(x-\xi_j(t))dx&=\xi'_j(t)\int_{\mathbb{R}}T^s(t,x)w''(x-\xi_j(t))dx \Leftrightarrow\\
\frac{d}{dt}\int_{\mathbb{R}}T^s w'(x-\xi_j(t))dx&=0\Leftrightarrow
\int_{\mathbb{R}}T^s w'(x-\xi_j(t))dx=c.
\end{align*}
Using the fact that $T^s(s,x)=0$ by above equality we deduce that $T^s$ satisfies the orthogonality condition \eqref{condpsi}.

Set
$$A_j^{(s,t)}=\left\{(\tau,x)\in(s,t)\times\mathbb{R}:\;\frac{\xi_j^0(\tau)+\xi_{j-1}^0(\tau)}{2}< x< \frac{\xi_j^0(\tau)+\xi_{j+1}^0(\tau)}{2}\right\},$$
with $\xi_0^0=-\infty,$ $\xi_{k+1}^0=\infty$ and
$$A_{j,R}^{(s,t)}=\left\{(\tau,x)\in(s,t)\times\mathbb{R}:\;|x-\xi_j^0(\tau)|<R+1\right\}.$$

We will prove \eqref{estfix} by contradiction. Let $\{s_i\},\;\{\overline{t}_i\}$ be sequences such that $s_i< \overline{t}_i\leq -T_0,$ and
$s_i\downarrow-
\infty,$ $\overline{t}_i\downarrow-\infty.$ We assume that there exists $g_i$ such that $g_i$ satisfies (\ref{cond}) and $\psi_i$  solves (\ref{mainpro*}) with $s=s_i,$ $-T=\overline{t}_i$ and $g=g_i.$ Finally we assume that
\be
\left|\left|\frac{\psi_i}{\xF}\right|\right|_{L^\infty((s_i,\overline{t}_i)\times\mathbb{R})}=1,\label{contr}
\ee
$$\left|\left|\frac{g_i}{\xF}\right|\right|_{L^\infty((s_i,\overline{t}_i)\times\mathbb{R})}\rightarrow0,$$

First we note that we can assume
$$s_i+1<\overline{t}_i.$$
Indeed let $\xl>0$ then the function $\psi_i=e^{\xl (t-s_i)}v_i(t,x)$ satisfies
\begin{align}\nonumber
&v_t=\; v_{xx}+(-\xl+f'(z(t,x)))v+e^{-\xl (t-s_i)}g(t,x),\quad \mathrm{in}\;\;\;(s,-T]\times\mathbb{R},\\
 &v(s,x)=0,\phantom{\psi_{xx}+f'(z)\psi+g(t,x)}\qquad \mathrm{in}\;\;\;\mathbb{R} \label{3333}.
\end{align}
Let $M>0$ be large enough. Set
$$\xf_j(t,x)=M\left(e^{\xs\left(-x+\xi_{j+1}^0(t)\right)}+e^{\xs\left(x-\xi_{j-1}^0(t)\right)}\right),$$
next observe that there exists $C>0$ independent of $t$ such that
\be
\xF(t,x)\leq C \xf_j(t,x),\quad\forall (t,x)\in(-\infty,-1)\times\mathbb{R}.\label{estPhi}
\ee
Now since $|f'(z(t,x)))|\leq C_0$ where $C_0$ does not depend on $t,$
we can choose $\xl>2C_0$ independent of $t$ such that the function $\xf_j$
satisfies for any $(t,x)\in (s,-1]\times\mathbb{R}$
$$(\xf_j)_t- (\xf_j)_{xx}+(\xl-f'(z(t,x)))\xf_j\geq c_1\xf_j(t,x)\geq c_2M\xF(t,x)\geq e^{-\xl (t-s_i)}g_i(t,x) ,$$
where $c_1, c_2>0$ are independent of $t$ and
$$M=\frac{1}{c_2}\left|\left|\frac{g_i}{\xF}\right|\right|_{L^\infty((s_i,\overline{t}_i)\times\mathbb{R})}.$$
Thus we can use $\xf_j$ like barrier to obtain
\be|v_i(t,x)|\leq \xf_j(t,x)\Rightarrow |\psi_i(t,x)|\leq C e^{\xl (t-s_i)}\xF(t,x)\left|\left|\frac{g_i}{\xF}\right|\right|_{L^\infty((s_i,\overline{t}_i)\times\mathbb{R})}.\label{4444}\ee
Thus by the above inequality we can choose $s_i+1<\overline{t}_i.$

To reach at contradiction we need the following assertion,

\textbf{Assertion 1.} \emph{Let $R>0$ then we have}
\be
\lim_{i\rightarrow\infty}\left|\left|\frac{\psi_i}{\xF}\right|\right|_{L^\infty(A_{j,R}^{(s_i,\overline{t}_i)})}=0,\;\;\;\;\;\foral j=1,...k.\label{contr2}
\ee
Let us first assume that \eqref{contr2} is valid. Set
$$
\xm_{i,j}:=\left|\left|\frac{\psi_i}{\xF}\right|\right|_{L^\infty(A_{j,R}^{(s_i,\overline{t}_i)})}\longrightarrow_{i\rightarrow\infty}0,\;\;\;\;\;\forall j=1,...k,
$$
and let
 $$\frac{\xi_j(t)+\xi_{j-1}(t)}{2}\leq x\leq \frac{\xi_j(t)+\xi_{j+1}(t)}{2},\;j=1,...,k,$$
with $\xi_0=-\infty$ and $\xi_{k+1}=\infty.$

If $n\leq j-1,$ then we have by our assumptions on $\xi_n$

$$|w(x-\xi_n(t))-1|\leq Ce^{\sqrt{2}(-x+\xi_{n-1}(t))}\leq Ce^{-\frac{\sqrt{2}}{2}(\xi_{j}-\xi_{j-1}(t))} \leq \frac{C}{\sqrt{t}}.$$
Similarly if $n\geq j+1$

$$|w(x-\xi_n(t))+1|\leq2e^{\sqrt{2}(x-\xi_{n+1}(t))}\leq \frac{C}{\sqrt{t}}.$$
Moreover if we assume that $|x-\xi_j(t)|>R+1,$ then we have that
$$|w(x-\xi_j(t))|\geq w(R).$$
Combining all above
for any $0<\xe<\sqrt{2}$ there exists $i_0\in \mathbb{N}$ and $R>0$ such that

\be
-f'(z(t,x))\geq 2-\xe^2,\;\;\forall t\leq \overline{t}_i,\;\;x\in \mathbb{R}\setminus\cup_{j=1}^kA_{j,R}^{(s_i,t_i)}\;\;\; \text{and}\;\;\;i\geq i_0.\label{fr11}
\ee

Consider the function

$$\overline{\xf}_{i,j}(t,x)=M\left(e^{\xs\left(-x+\xi_{j+1}^0(t)\right)}+e^{\xs\left(x-\xi_{j-1}^0(t)\right)}\right)
\left(\left|\left|\frac{g_i}{\xF}\right|\right|_{L^\infty((s_i,\overline{t}_i)\times\mathbb{R})}
+\sup_{1\leq j\leq k}\xm_{i,j}\right),$$
where $M>1$  is large enough which does not depend on $s_i, \overline{t}_i.$

Let $\xe>0$ be such that $2-\xe^2>\xs^2.$ Then we can choose $i_0$ such that for any $i> i_0$ and
$\forall\; (t,x)\in((s_i,t_i)\times\mathbb{R})\setminus\cup_{j=1}^kA_{j,R}^{(s_i,\overline{t}_i)},$ the function $\overline{\xf}_{i,j}(t,x)$ satisfies
\be(\overline{\xf}_{i,j})_t- (\overline{\xf}_{i,j})_{xx}-f'(z(t,x))\overline{\xf}_{i,j}\geq c_1\overline{\xf}_{i,j}\geq c_2M\xF(t,x)\geq g_i(t,x), \label{5555}\ee
where the constants $0<c_2<c_1<1$ are independent of $t,$ $M\geq\frac{1}{c_2}$ and we have used \eqref{estPhi}.

Let $0\leq\eta\leq1$ be a smooth function in $C_0^\infty\mathbb{R}$ such that $\eta=1$ if $|x|<1$ and $\eta=2$ if $|x|>2.$
Set $\xz=\eta^2 (\frac{t}{R})\max(\psi_i(t,x)-\xf_{i,j}(t,x),0).$

Note that by \eqref{4444} we can choose $M>0$ such that
$$\max(\psi_i(t,x)-\xf_{i,j}(t,x),0)=0,\quad\forall\; (t,x)\in \overline{\cup_{j=1}^kA_{j,R}^{(s_i,\overline{t}_i)}},$$
thus by \eqref{fr11} and \eqref{5555} we can easily obtain

\begin{align*}\int_{s_i}^{\overline{t}_i}\int_{\mathbb{R}}(\psi_i-\xf_{i,j})_t\xz dxdt&+\int_{s_i}^{\overline{t}_i}\int_{\mathbb{R}}(\psi_i-\xf_{i,j})_{x}\xz_x dxdt\\
&+(2-\xe^2)\int_{s_i}^{\overline{t}_i}\int_{\mathbb{R}}(\psi_i-\xf_{i,j})\xz dxdt\leq0.
\end{align*}

By the last inequality and by standard arguments we obtain
$$|\psi_i(t,x)|\leq |\overline{\xf}_{i,j}(t,x)|,\;\;\forall(t,x)\in ((s_i,\overline{t}_i)\times\mathbb{R})\setminus\cup_{j=1}^kA_{j,R}^{(s_i,\overline{t}_i)},\;\;j=1,...k,\;\;\;i\geq i_0.$$
Thus we have
$$|\psi_i(t,x)|\leq |\overline{\xf}_{i,j}(t,x)|,\;\;\forall(t,x)\in (s_i,\overline{t}_i)\times\mathbb{R},\;\;j=1,...k.$$
Hence by \eqref{3333} we can easily obtain that
$$1=\left|\left|\frac{\psi_i}{\xF}\right|\right|_{L^\infty((s_i,\overline{t}_i)\times\mathbb{R})}\leq
M\left(\left|\left|\frac{g_i}{\xF}\right|\right|_{L^\infty((s_i,\overline{t}_i)\times\mathbb{R})}
+\sup_{1\leq j\leq k}\xm_{i,j}\right),$$
which is clearly a contradiction if we choose $i$ large enough.

\noindent\textbf{Proof of Assertion 1.}
We will prove Assertion 1 by contradiction in four steps.

Let us give first the contradict argument and some notations.  We assume that (\ref{contr2}) is not valid. Then there exists $j\in\{1,...,k\}$ and $\xd>0$ such that
$$\left|\left|\frac{\psi_i}{\xF}\right|\right|_{L^\infty(A_{j,R}^{(s_i,\overline{t}_i)})}>\xd>0,\;\;\forall i\in\mathbb{N}.$$
Let $(t_i,y_i)\in A_{j,R}^{(s_i,\overline{t}_i)}$ such that
\be
\left|\frac{\psi_i(t_i,y_i)}{\xF(t_i,y_i)}\right|>\xd.\label{667}
\ee

We observe here that by definition of $\xF$
\be
\xF(t_i,y_i)=e^{\xs(-y_i+\xi_{j-1}(t_i))}+e^{\xs(y_i-\xi_{j+1}(t_i))}.\label{obs1}
\ee

We set $y=x+\xi_j(t+t_i),\;\;y_i=x_i+\xi_j(t_i)$ and
$$\xf_i(t,x)=\frac{\psi_i(t+t_i,x+x_i+\xi_j(t+t_i))}{\xF(t_i,x_i+\xi_j(t_i))}.$$

Then $\xf_i$ satisfies
\bea\nonumber
(\xf_i)_t&=&\; (\xf_i)_{xx}-\xi_j'(t+t_i)(\xf_i)_x+f'(z(t+t_i,x+x_i+\xi_j(t+t_i)))\xf_i\\ \nonumber&+&\frac{g_i(t+t_i,x+x_i+\xi_j(t+t_i))}{\xF(t_i,x_i+\xi_j(t_i))}, \quad \mathrm{in}\;\;\;(s_i-t_i,0]\times\mathbb{R},\\
\xf_i(s_i-t_i,x)&=&0,\phantom{\psi_{xx}+f'(z)\psi+g(t,x)}\qquad \mathrm{in}\;\;\;\mathbb{R}.\label{eq11}
\eea

Also set
\begin{align}
\nonumber
&B_{t_i,n,j}=\Bigg\{(t,x)\in(s_i-t_i,0]\times\mathbb{R}:
\;\frac{\xi_n^0(t+t_i)+\xi_{n-1}^0(t+t_i)}{2}-\xi_j(t+t_i)-x_i\\ \nonumber
&\leq x\leq
 \frac{\xi_n^0(t+t_i)+\xi_{n+1}^0(t+t_i)}{2}-\xi_j(t+t_i)-x_i\Bigg\}
\end{align}
and
\begin{align*}
B_{t_i,n,j}^M=B_{t_i,n,j}\cap\left\{(t,x)\in(s_i-t_i,0]\times\mathbb{R}: |x+\xi_j(t+t_i)+x_i-\xi_n^0(t+t_i)|>M\right\},
\end{align*}
where $n=1,....,k$ and $M>0.$
We note here that $|x_i|<R+1, \;\forall\; i\in\mathbb{N},$ $|\xf_i(0,0)|=\left|\psi_i(t_i,y_i)/\xF(t_i,y_i)\right|>\xd>0.$
Also in view of the proof of \eqref{4444} and the assumption \eqref{667} we can assume that
$$\liminf t_i-s_i>\infty.$$
Without loss of generality we assume that $x_i\rightarrow x_0\in B_{R+1}(0),$ $\lim_{i\rightarrow\infty} t_i-s_i=\infty$ (otherwise take a subsequence).

\medskip
\noindent\textbf{Step 1}

We assert that
$\xf_i\rightarrow\xf$ locally uniformly, $\xf(0,0)>\xd$  and $\xf$ satisfies
\be
\xf_t=\; \xf_{xx} +f'(w(x+x_0))\xf,\qquad \mathrm{in}\;\;\;(-\infty,0]\times\mathbb{R}.\label{equ1}
\ee

Let $(t,x)\in B_{t_i,n,j},\;1\leq n\leq k.$ By \eqref{contr} and \eqref{obs1} we have that
\begin{align}\nonumber
&|\xf_i(t,x)|\leq \left|\frac{\psi_i(t+t_i,x+x_i+\xi_j(t+t_i))}{\xF(t_i,x_i+\xi_j(t_i))}\right|\\ \nonumber
&\leq \left|\frac{\xF(t+t_i,x+x_i+\xi_j(t+t_i))}{\xF(t_i,x_i+\xi_j(t_i))}\right|\\ \nonumber
&= \frac{e^{\xs(-(x+x_i+\xi_j(t+t_i))+\xi_{n-1}(t+t_i))}+
e^{\xs(x+x_i+\xi_j(t+t_i)-\xi_{n+1}(t+t_i))}}{e^{\xs(-(x_i+\xi_j(t_i))+\xi_{j-1}(t_i))}+e^{\xs(x_i+\xi_j(t_i)-\xi_{j+1}(t_i))}}\\ \nonumber
&\leq e^{\xs(-(x+x_i+\xi_j(t+t_i)-\xi_{n}(t+t_i))-(\xi_{n}(t+t_i)-\xi_{n-1}(t+t_i))+x_i+(\xi_j(t_i))-\xi_{j-1}(t_i))}\\ \nonumber
&+e^{\xs((x+x_i+\xi_j(t+t_i)-\xi_{n}(t+t_i))+(\xi_{n}(t+t_i)-\xi_{n+1}(t+t_i))+x_i+(\xi_{j+1}(t_i))-\xi_{j}(t_i))}\\ \nonumber
&\leq C_0(\xb,||h||_{L^\infty},\sup_{1\leq j\leq k}|\xg_j|,\xs,R)\\ &\times\left(\frac{t_i}{t+t_i}\right)^{\frac{\xs}{\sqrt{2}}}e^{\xs|x+\xi_j(t+t_i)-\xi_n(t+t_i)|},\quad\forall i\in\mathbb{N},\;(t,x)\in B_{t_i,n,j},
\label{constant}
\end{align}
where in the last inequality we have used \eqref{ksij}.

Now note here that $$\cup_{i=1}^\infty B_{t_i,j,j}=(-\infty,0]\times\mathbb{R}.$$

Thus the proof of the assertion of this step is complete.

\medskip
\noindent\textbf{Step 2} In this step we prove the following orthogonality condition for $\xf$

\be
\int_{\mathbb{R}}\xf(t,x)w'(x+x_0)dx=0,\;\;\;\forall t\in (-\infty,0].\label{oth}
\ee

Let $t\in \cap_{i=i_0}^\infty(s_i-t_i,0],$ for some $i_0\in\mathbb{N},$ and

\begin{align}
\nonumber
x\in &B_{t,t_i,n,j}=\Bigg\{x\in\mathbb{R}:
\;\frac{\xi_n^0(t+t_i)+\xi_{n-1}^0(t+t_i)}{2}-\xi_j(t+t_i)-x_i\\ \nonumber
&\leq x\leq
 \frac{\xi_n^0(t+t_i)+\xi_{n+1}^0(t+t_i)}{2}-\xi_j(t+t_i)-x_i\Bigg\}.
\end{align}
By (\ref{constant}) we have that
\be
\left|\int_{B_{t,t_i,j,j}}\xf_i(t,x)w'(x+x_i)dx\right|\leq C_0\int_{\mathbb{R}}e^{-(\sqrt{2}-\xs)|x|}dx< C.\label{est2}
\ee

Let $n>j,$ then there exists $i_0$ such that for any $i>i_0$ we have $x>0.$
Also by \eqref{constant}, the assumptions on $\xi$ (see Notation \ref{notation}) and the fact that $|x|<R+1$ we have that

\begin{align}\nonumber
&\left|\int_{B_{t,t_i,n,j}}\xf_i(t,x)w'(x+x_i)dx\right|\\ \nonumber
&\leq C_0 \int_{{\frac{\xi_n^0(t+t_i)+\xi_{n-1}^0(t+t_i)}{2}-\xi_j(t+t_i)-x_i}}^{\frac{\xi_n^0(t+t_i)+
\xi_{n+1}^0(t+t_i)}{2}-\xi_j(t+t_i)-x_i}e^{-\sqrt{2}x+\xs|x+\xi_j(t+t_i)-\xi_n(t+t_i)|}dx\\ \nonumber
&= C_0 e^{-\sqrt{2}(\xi_n(t+t_i)-\xi_j(t+t_i))} \int_{{\frac{\xi_n^0(t+t_i)+\xi_{n-1}^0(t+t_i)}{2}-\xi_n(t+t_i)-x_i}}^{\frac{\xi_n^0(t+t_i)+
\xi_{n+1}^0(t+t_i)}{2}-\xi_n(t+t_i)-x_i}e^{-\sqrt{2}y+\xs|y|}dy.
\end{align}
Now
\begin{align*}
&\int_{{\frac{\xi_n^0(t+t_i)+\xi_{n-1}^0(t+t_i)}{2}-\xi_n(t+t_i)-x_i}}^{\frac{\xi_n^0(t+t_i)+
\xi_{n+1}^0(t+t_i)}{2}-\xi_n(t+t_i)-x_i}e^{-\sqrt{2}y+\xs|y|}dy\\
&=\int_{0}^{\frac{\xi_n^0(t+t_i)+
\xi_{n+1}^0(t+t_i)}{2}-\xi_n(t+t_i)-x_i}e^{-\sqrt{2}y+\xs|y|}dy\\
&+\int_{{\frac{\xi_n^0(t+t_i)+\xi_{n-1}^0(t+t_i)}{2}-\xi_n(t+t_i)-x_i}}^{0}e^{-\sqrt{2}y+\xs|y|}dy\\
&\leq C \left(e^{-\frac{\sqrt{2}-\xs}{2}(\xi_{n+1}^0(t+t_i)-\xi_{n}^0(t+t_i))}+e^{\frac{\sqrt{2}+\xs}{2}(\xi_{n}^0(t+t_i)-\xi_{n-1}^0(t+t_i))}\right) +C_1\\
&\leq C\left(e^{-\frac{\sqrt{2}-\xs}{2\sqrt{2}}(\log(-2\sqrt{2}\xb (t+ t_i))}+e^{\frac{\sqrt{2}+\xs}{2\sqrt{2}}(\log(-2\sqrt{2}\xb (t+ t_i)))}\right)+C_1.
\end{align*}
But
$$e^{-\sqrt{2}(\xi_n(t+t_i)-\xi_j(t+t_i))}\leq Ce^{-\sqrt{2}(\xi_n(t+t_i)-\xi_{n-1}(t+t_i))}\leq Ce^{-\log(-2\sqrt{2}\xb (t+ t_i))}.$$
Thus combining all above we have
\be
\left|\int_{B_{t,t_i,n,j}}\xf_i(t,x)w'(x+x_i)dx\right|\leq C e^{-\frac{(\sqrt{2}-\xs)}{2\sqrt{2}}\log(-2\sqrt{2}\xb (t+ t_i))}\rightarrow_{i\rightarrow\infty}0.\label{est1}
\ee
Similarly the estimate \eqref{est1} is valid if $n<j.$
Now note that

\bea\nonumber
&&\int_{\mathbb{R}}\xf_i(t,x)w'(x+x_i)dx=\int_{\mathbb{R}}\frac{\psi_i(t+t_i,x+x_i+\xi_j(t+t_i))}{\xF(t_i,x_i+\xi_j(t_i))}w'(x+x_i)dx\\
&=&\frac{1}{\xF(t_i,x_i+\xi_j(t_i))}\int_{\mathbb{R}}\psi_i(t+t_i,x)w'(x-\xi_j(t+t_i))dx=0.\label{3laloun}
\eea
By (\ref{est1}), (\ref{est2}) and \eqref{3laloun} we have that
$$0=\int_{\mathbb{R}}\xf_i(t,x)w'(x+x_i)dx\rightarrow\int_{\mathbb{R}}\xf(t,x)w'(x+x_0)dx$$
and the proof of this assertion follows.

\medskip
\noindent\textbf{Step 3} In this step we prove the following assertion:

 \emph{There exists $C=C(R,\xs)>0,$ such that}

\be
|\xf(t,x)|\leq Ce^{-\xs|x|},\;\;\forall(t,x)\in(-\infty,0]\times \mathbb{R}.\label{bound1}
\ee

Now, note that if $(t,x)\in B_{t_i,n,j},$ by definition of $\xi$ (Notation \ref{notation}), we have
$$
e^{\xs|x+\xi_j(t+t_i)-\xi_n(t+t_i)|}\leq C_0(\xb,||h||_{L^\infty},\sup_{1\leq j\leq k}|\xg_j|,\xs,R) e^{\xs|x|}.
$$
Thus, in view of the proof of \eqref{constant} we have that
$$
\left|\frac{g_i(t+t_i,x+x_i+\xi_j(t+t_i))}{\xF(t_i,x_i+\xi_j(t_i))}\right|\leq C\left|\left|\frac{g_i}{\xF}\right|\right|_{L^\infty((s_i,\overline{t}_i)\times\mathbb{R})}
e^{\xs|x|},\;\;\forall i\in\mathbb{N}.
$$
In view of the proof of Asertion 1 we can find $i_0$ and $M>0$ such that we use $$G(t,x)=M\left(e^{-\xs|x|}+\left|\left|\frac{g_i}{\xF}\right|\right|_{L^\infty((s_i,\overline{t}_i)\times\mathbb{R})}
e^{\xs|x|}\right),$$
as barrier, to prove
$$|\xf_i(t,x)|\leq G(t,x),\quad\forall (t,x)\in B_{t_i,j,j}^M\;\;\forall i\geq i_0.$$
And the proof of \eqref{bound1} follows if we send $i\rightarrow\infty.$

\medskip
\noindent\textbf{ Step 4}
Here  we prove the assertion \eqref{contr2}.

If we multiply (\ref{equ1}) by $\xf$ and integrate with respect $x$ we have by Proposition \ref{otinanai}
\bea\nonumber
0&=&\frac{1}{2}\int_{\mathbb{R}}(\xf^2)_tdx+\int_{\mathbb{R}}|\xf_x|^2-f'(w(x))|\xf|^2dx\\ \nonumber
&\geq&\frac{1}{2}\int_{\mathbb{R}}(\xf^2)_tdx+c\int_{\mathbb{R}}|\xf(t,x)|^2dx.
\eea

Set $a(t)=\int_{\mathbb{R}}|\xf(t,x)|^2dx,$ we have that there exists a $c_0$ such that
$$a'(t)\leq-c_0a(t)\Rightarrow a(t)>a(0)e^{c_0|t|},$$
which is a contradiction since
$$
||e^{\xs|x|}\xf||_{L^\infty((-\infty,-T_0-\widetilde{t}_0)\times\mathbb{R})}<C.
$$

\end{proof}
The following Proposition is well known, we give a proof for the convenience of the reader.
\begin{proposition}
Consider the Hilbert space $$H=\{\xz\in H^1(\mathbb{R}):\;\int_{\mathbb{R}}\xz(x)w'(x)dx=0\}.$$
Then the following inequality is valid
\be
\int_{\mathbb{R}}|\xz'(x)|^2-f'(w(x))|\xz|^2dx\geq c\int_{\mathbb{R}}|\xz(x)|^2dx, \qquad\forall\xz\in H\cap L^2(\mathbb{R}).\label{Heq}
\ee\label{otinanai}
\end{proposition}
\begin{proof}
Let $\xz\in H.$ Set $\xz=w'\xf.$ Then
\begin{align*}
\int_{\mathbb{R}}|\xz'(x)|^2-f'(w(x))|\xz|^2dx&=\int_{\mathbb{R}}|w''|^2|\xf|^2dx+\int_{\mathbb{R}}|w'|^2|\xf'(x)|^2\\
&+\int_{\mathbb{R}}w''w'(\xf^2)'dx-\int_{\mathbb{R}}f'(w(x))|w'\xf|^2dx\\
&=\int_{\mathbb{R}}|w'|^2|\xf'(x)|^2\geq0.
\end{align*}
Thus

$$\int_{\mathbb{R}}|\xz'(x)|^2-f'(w(x))|\xz|^2dx=0\quad\text{if and only if}\quad \xz=cw',$$
for some constant $c,$ which implies that $\xz=0.$

Now we assume that there exists a sequence $\{\xf_n\}_{n=1}^\infty\in H$ such that
$$\int_{\mathbb{R}}\xf_n^2dx=1$$
and
\be
\int_{\mathbb{R}}|\xf_n'(x)|^2-f'(w(x))|\xf_n|^2dx\leq \frac{1}{n}.\label{34}
\ee

Thus $\xf_n\rightharpoonup \xf$ in $H$ and $\xf_n\rightharpoonup \xf$ in $L^2(K)$ for any compact subset of $\mathbb{R}.$ Which implies

$$0=\int_{\mathbb{R}}\xf_n(x)w'dx\rightarrow\int_{\mathbb{R}}\xf w'dx=0,$$
$\xf\in H$
and
$$
\int_{\mathbb{R}}|\xf'(x)|^2-f'(w(x))|\xf|^2dx=0.
$$
Thus $\xf=0.$

But by \eqref{34} we have
$$2=2\int_{\mathbb{R}}|\xf|^2dx\leq 3\int_{\mathbb{R}}(1-w^2)|\xf|^2dx,$$
which implies that $\xf\neq0,$ which is clearly a contradiction

\end{proof}
\subsection{The problem \eqref{mainpro*} with $g(t,x)=h(t,x)-\sum_{j=1}^k c_i(t)w'(x-\xi_j(t))$}
In this subsection, we study the following problem.

\begin{align}\nonumber
\psi_t&= \psi_{xx}+f'(z(t,x))\psi+h(t,x)-\sum_{j=1}^k c_i(t)w'(x-\xi_j(t)),\quad \mathrm{in}\;\;\;(s,-T_0]\times\mathbb{R},\\
\psi(s,x)&=0,\phantom{\xD\psi+f'(z)\psi+f(t,x)-\;\;\; c_i(t)w'(x-\xi_j(t))}\quad \mathrm{in}\;\;\;\mathbb{R},\label{proci}
\end{align}
where $h\in\mathcal{C}_\xF((s,-T)\times\mathbb{R})$ and $c_i(t)$ satisfies the following (nearly diagonal) system
\begin{align}\nonumber
&\sum_{i=1}^kc_i(t)\int_{\mathbb{R}}w'(x-\xi_i(t))w'(x-\xi_j(t))dx\\ \nonumber
&=\int_{\mathbb{R}}\left( \psi_{xx}(t,x) w'(x-\xi_j(t))+f'(z(t,x)\psi(t,x)\right) dx\\ \nonumber
&-\xi'_j(t)\int_{\mathbb{R}}\psi(t,x)w''(x-\xi_j(t))dx\\
&+\int_{\mathbb{R}}h(t,x)w'(x-\xi_j(t))dx,\qquad\forall i=1,...,k,\;t<-T.\label{ci(t)}
\end{align}
We note here that if $\psi$ is a solution of \eqref{proci} and $c_i(t)$ satisfies the above system then $g(t,x)=h(t,x)-\sum_{j=1}^k c_i(t)w'(x-\xi_j(t))$ satisfies \eqref{cond}. Thus in view of the proof of \eqref{condpsi} we have that $\psi$ satisfies the orthogonality conditions $$\int_{\mathbb{R}}\psi(t,x)w'(x-\xi_i(t))dx=0,\quad\forall i=1,...,k,\;s<t<-T_0.$$

\medskip
The main result of this subsection is the following
\begin{lemma}
Let $h\in\mathcal{C}_\xF((s,-T)\times\mathbb{R})$. Then there exist a uniform constant $T_0\geq T>0,$ and a unique solution  $T^s$ of the problem \eqref{proci}.

Furthermore, we have that $T^s$ satisfies the orthogonality conditions (\ref{orthcond}), $\forall s<t<-T_0,$ and the following estimate
 \be
||T^s||_{\mathcal{C}_\xF((s,t)\times\mathbb{R})}\leq C||h||_{\mathcal{C}_\xF((s,t)\times\mathbb{R})},\label{estfixmef}
\ee
where $C>0$ is a uniform constant.\label{cilemma*}
\end{lemma}
To prove the above Lemma we need the following result.

\begin{lemma}
Let $T>0$ big enough, $h\in\mathcal{C}_\xF((s,-T)$ and $\psi\in \mathcal{C}_\xF((s,-T)\times\mathbb{R}).$ Then there exist $c_i(t),\;i=1,...,k$ such that the nearly diagonal system \eqref{ci(t)} holds.

Furthermore the following estimates for $c_i$ are valid, for some constant $C>0$ that does not depends on $T,\;s,\;t,\;\psi,\;f$
$$|c_i(t)|\leq C\left(\left(\frac{1}{|t|}\right)^{1+\frac{\xs}{2\sqrt{2}}}\left|\left|\frac{\psi}{\xF}\right|\right|_{L^\infty((s,-T)\times\mathbb{R})}
+\left(\frac{1}{|t|}\right)^{\frac{\xs}{\sqrt{2}}}\left|\left|\frac{h}{\xF}\right|\right|_{L^\infty((s,-T)\times\mathbb{R})}\right),$$
\hfill$\forall\;t\in [s,-T],\;\;\;\forall\; i=1,...,k,$
\begin{align}\nonumber
\left|\frac{c_i(t)w'(x-\xi_i(t))}{\xF(t,x)}\right|\leq C\left(\left(\frac{1}{|t|}\right)^{1-\frac{\xs}{2\sqrt{2}}}\left|\left|\frac{\psi}{\xF}\right|\right|_{L^\infty((s,-T)\times\mathbb{R})}
+\left|\left|\frac{h}{\xF}\right|\right|_{L^\infty((s,-T)\times\mathbb{R})}\right),
\end{align}
\hfill$\forall\;t\in [s,-T],\;\;\;\forall\; i=1,...,k.$
\label{cilemma}
\end{lemma}
\begin{proof}
For $i< j,$ we have
\begin{align}\nonumber
&\int_{\mathbb{R}}w'(x-\xi_i(t))w'(x-\xi_j(t))dx=\int_{\mathbb{R}}w'(x+(\xi_j(t)-\xi_i(t)))w'(x)dx\\ \nonumber
&=C\int_{{\mathbb{R}}}\left(\frac{1}{e^{\frac{\sqrt{2}}{2}\left(x-(\xi_j(t)-\xi_i(t))\right)}+e^{\frac{\sqrt{2}}{2}\left(-x+(\xi_j(t)-\xi_i(t))\right)}}\right)^2
\left(\frac{1}{e^{\frac{\sqrt{2}}{2}x}+e^{-\frac{\sqrt{2}}{2}x}}\right)^2dx\\ \nonumber
&=C\frac{1}{e^{\sqrt{2}(\xi_j(t)-\xi_i(t))}}
\int_{{\mathbb{R}}}\left(\frac{1}{e^{\frac{\sqrt{2}}{2}\left(x-2(\xi_j(t)-\xi_i(t))\right)}+e^{-\frac{\sqrt{2}}{2}x}}\right)^2
\left(\frac{1}{e^{\frac{\sqrt{2}}{2}x}+e^{-\frac{\sqrt{2}}{2}x}}\right)^2dx\\ \nonumber
&=C\frac{1}{e^{\sqrt{2}(\xi_j(t)-\xi_i(t))}}\int_{{\mathbb{R}}}F(t,x)dx,
\end{align}
where
$$F(t,x)=\left(\frac{1}{e^{\frac{\sqrt{2}}{2}\left(x-2(\xi_j(t)-\xi_i(t))\right)}+e^{-\frac{\sqrt{2}}{2}x}}\right)^2
\left(\frac{1}{e^{\frac{\sqrt{2}}{2}x}+e^{-\frac{\sqrt{2}}{2}x}}\right)^2.$$
Now
\begin{align}\nonumber
\int_{2(\xi_j(t)-\xi_i(t))}^\infty F(t,x)dx&<C,\\ \nonumber
\int_{-\infty}^0 F(t,x)dx&<C,\\ \nonumber
\int^{2(\xi_j(t)-\xi_i(t))}_0 F(t,x)dx&\leq C((\xi_j(t)-\xi_i(t))+1),
\end{align}
where the constant $C>0$ does not depend on $t.$

Thus we can easily obtain
$$
\int_{\mathbb{R}}w'(x-\xi_i(t))w'(x-\xi_j(t))dx\leq C\frac{|\log|t||}{t},\qquad\forall i\neq j,\;i,j=1,...k,
$$
where in the above inequality we have used the assumptions on $"\xi_j"$ see (Notation \ref{notation}). Thus the system is nearly diagonal and we can solve it for $T$ big enough.

Also we have
\begin{align}\nonumber
\int_{-\infty}^\infty\xF(t,x)dx=\sum_{j=1}^{k}\int_{\frac{\xi_j^0(t)+\xi_{j-1}^0(t)}{2}}^{\frac{\xi_j^0(t)
+\xi_{j+1}^0(t)}{2}}&e^{\xs(-x+\xi_{j-1}^0(t))}+e^{\xs(x-\xi_{j+1}^0(t))}dx\\
&\leq C\left(\frac{1}{|t|}\right)^\frac{\xs}{2\sqrt{2}},\label{****}
\end{align}
where $\xi_0^0=-\infty,\;\xi_{k+1}^0=\infty.$

\begin{align}
\nonumber
&\left|\int_{\mathbb{R}}\left(\psi_{xx}+f'(z(t,x))\psi\right)w'(x-\xi_j(t))dx\right|\\ \nonumber
&=\left|\int_{\mathbb{R}}\left(f'(w(x-\xi_j(t)))-f'(z(t,x))\right)\psi(t,x)w'(x-\xi_j(t))dx\right|\\ \nonumber
&=\left|\int_{\mathbb{R}}\left(f'(w(x))-f'(z(t,x+\xi_j(t)))\right)\psi(t,x+\xi_j(t))w'(x)dx\right|\\ \nonumber
&\leq C\left|\left|\frac{\psi}{\xF}\right|\right|_{L^\infty((s,-T)\times\mathbb{R})}\\ \nonumber
&\qquad\times\int_{\mathbb{R}}|(-1)^{j+1}w(x)-z(t,x+\xi_j(t))|\xF(t,x+\xi_j(t))w'(x)dx\\ \nonumber
&\leq C\left|\left|\frac{\psi}{\xF}\right|\right|_{L^\infty((s,-T)\times\mathbb{R})}\frac{1}{|t|}\int_{\mathbb{R}}\xF(t,x+\xi_j(t))dx\\
&\leq C\left|\left|\frac{\psi}{\xF}\right|\right|_{L^\infty((s,-T)\times\mathbb{R})}\left(\frac{1}{|t|}\right)^{1+\frac{\xs}{2\sqrt{2}}}.
\label{estci}
\end{align}
In the last inequality we have used the fact that, if $i>j,$ then
$$|w(x+\xi_j-\xi_i)+1|w'(x)\leq C \frac{1}{e^{\sqrt{2}(\xi_j(t)-\xi_i(t))}}.$$
Similarly we have that
\be
\left|\xi'_j(t)\int_{\mathbb{R}}\psi(t,x)w''(x-\xi_j(t))dx\right|\leq C\left|\left|\frac{\psi}{\xF}\right|\right|_{L^\infty((s,-T)\times\mathbb{R})} \left(\frac{1}{|t|}\right)^{1+\frac{\xs}{2\sqrt{2}}},\label{estci1}
\ee
$$
\left|\int_{\mathbb{R}}h(t,x)w'(x-\xi_j(t))dx\right|\leq C\left|\left|\frac{h}{\xF}\right|\right|_{L^\infty((s,-T)\times\mathbb{R})}\left(\frac{1}{|t|}\right)^{\frac{\xs}{\sqrt{2}}}.
$$
Thus, by the above inequalities we have
$$|c_i(t)|\leq C\left|\left|\frac{\psi}{\xF}\right|\right|_{L^\infty((s,-T)\times\mathbb{R})} \left(\frac{1}{|t|}\right)^{1+\frac{\xs}{2\sqrt{2}}},\;\;\;\forall\; i=1,...,k.$$
Now if $\frac{\xi_{i-1}^0(t)+\xi_i^0(t)}{2}\leq x\leq\frac{\xi_{i+1}^0(t)+\xi_i^0(t)}{2} ,$ we have
\begin{align}\nonumber
&\left|\frac{c_i(t)w'(x-\xi_i(t))}{\xF(t,x)}\right|\\ \nonumber&
\leq
 C|t|^{\frac{\xs}{\sqrt{2}}}\left(\left(\frac{1}{|t|}\right)^{1+\frac{\xs}{2\sqrt{2}}}\left|\left|\frac{\psi}{\xF}\right|\right|_{L^\infty((s,-T)\times\mathbb{R})}
+\left(\frac{1}{|t|}\right)^{\frac{\xs}{\sqrt{2}}}\left|\left|\frac{h}{\xF}\right|\right|_{L^\infty((s,-T)\times\mathbb{R})}\right)\\
&\leq  C\left(\left(\frac{1}{|t|}\right)^{1-\frac{\xs}{2\sqrt{2}}}\left|\left|\frac{\psi}{\xF}\right|\right|_{L^\infty((s,-T)\times\mathbb{R})}
+\left|\left|\frac{h}{\xF}\right|\right|_{L^\infty((s,-T)\times\mathbb{R})}\right).\label{12}
\end{align}
Combining all above the proof of Lemma is complete.

\end{proof}

\emph{Proof of Lemma \ref{cilemma*}.}
First we recall that
\begin{align}\nonumber
\sum_{i=1}^kc_i(t)&\int_{\mathbb{R}}w'(x-\xi_i(t))w'(x-\xi_j(t))dx\\ \nonumber
&=\int_{\mathbb{R}}\left(\psi_{xx}(t,x) w'(x-\xi_j(t))+f'(z(t,x)\psi(t,x)\right)\psi dx\\ \nonumber
&-\xi'_j(t)\int_{\mathbb{R}}\psi(t,x)w''(x-\xi_j(t))dx\\ \nonumber
&+\int_{\mathbb{R}}h(t,x)w'(x-\xi_j(t))dx,\qquad\forall i=1,...,k,\;t<-T.
\end{align}

We will prove that there exists a unique solution of the problem (\ref{proci}) by using a fix point argument.

Let $$X^s=\{\psi:\;||\psi||_{\mathcal{C}_\xF((s,s+1)\times\mathbb{R})}<\infty\}.$$
We consider the operator $A^s: X^s\rightarrow X^s$ given by
$$A^s(\psi)=T^s(h-C(\psi)),$$
where $T^s(g)$ denotes the solution to (\ref{mainpro*}) and $C(\psi)=\sum_{j=1}^k c_i(t)w'(x-\xi_j(t)).$
Also by standard parabolic estimates we have
\be
||A^s(\psi)||_{\mathcal{C}_\xF((s,s+1)\times\mathbb{R})}\leq C_0\left(
||h-C(\psi)||_{\mathcal{C}_\xF((s,s+1)\times\mathbb{R})}\right),\label{lem1}
\ee
for some uniform constant $C_0>0.$
We will show that the map $A^s$ defines a contraction mappping and we will apply the fixed point theorem to it.
To this end, set $c=C_0||h||_{\mathcal{C}_\xF((s,-T)\times\mathbb{R})}$ and
$$X^s_c=\{\psi:\;||\psi||_{C_\xF((s,s+1)\times\mathbb{R})}<2c\},$$
where constant $C_0$ taken from  \eqref{lem1}, for $C(T,s)=C(s+1,s).$ We note here that by standard parabolic theory, the constant $C(T,s)= C_0|(-T-s)|.$

We claim that $A^s(X^s_c)\subset X^s_c,$ indeed by inequality \eqref{lem1} we have
\bea
\nonumber
&&||A^s(\psi)||_{\mathcal{C}_\xF((s,s+1)\times\mathbb{R})}\leq C_0\left(
||h-C(\psi)||_{\mathcal{C}_\xF((s,s+1)\times\mathbb{R})}\right)\\ \nonumber
&\leq& C_0\left(||h||_{\mathcal{C}_\xF((s,-T)\times\mathbb{R})}+
||C(\psi)||_{\mathcal{C}_\xF((s,s+1)\times\mathbb{R})}\right)\\ \nonumber
&\leq&
\frac{C_0}{\sqrt{|s+1|}}\left(||\psi||_{\mathcal{C}_\xF((s,s+1)\times\mathbb{R})}\right)+c\\ \nonumber
&\leq& c+c,
\eea
where in the above inequalities we have used Lemma \ref{cilemma} and we have chosen $|s|$ big enough. Next we show that $A^s$ defines a contraction map. Indeed, since $C(\psi)$ is linear in $\psi$ we have
\bea\nonumber
&&||A^s(\psi_1)-A^s(\psi_2)||_{\mathcal{C}_\xF((s,s+1)\times\mathbb{R})}\\ \nonumber
&\leq&||C(\psi_1)-C(\psi_2)||_{\mathcal{C}_\xF((s,s+1)\times\mathbb{R})}=
||C(\psi_1-\psi_2)||_{\mathcal{C}_\xF((s,s+1)\times\mathbb{R})}\\ \nonumber
&\leq&\frac{C}{\sqrt{|s+1|}}||(\psi_1-\psi_2)||_{\mathcal{C}_\xF((s,s+1)\times\mathbb{R})}\\ \nonumber
&\leq&\frac{1}{2}||(\psi_1-\psi_2)||_{\mathcal{C}_\xF((s,s+1)\times\mathbb{R})}.
\eea
Combining all above, we have by fixed point theorem that there exists a $\psi^s\in X^s$ so that $A^s(\psi^s)=\psi^s,$ meaning that the equation (\ref{proci}) has a solution $\psi^s,$ for $-T=s+1.$

We claim that $\psi^s(t,x)$ can be extended to a solution on $(s,-T_0]\times\mathbb{R},$ still satisfies the orthogonality condition (\ref{orthcond}) and the a priori estimate. To this end, assume that our solution $\psi^s(t,\cdot)$ exists for $s\leq t\leq -T,$ where $T>T_0$ is the maximal time of the existence. Since $\psi^s$ satisfies the orthogonality condition (\ref{orthcond}), we have by \eqref{estfix}
$$||\psi^s||_{\mathcal{C}_\xF((s,-T)\times\mathbb{R})}\leq C||h-C(\psi)||_{\mathcal{C}_\xF((s,-T)\times\mathbb{R})}.$$
Thus if we choose $T_0$ big enough, we have by Lemma \ref{cilemma} that
$$||\psi^s||_{\mathcal{C}_\xF((s,-T)\times\mathbb{R})}\leq C||h||_{\mathcal{C}_\xF((s,-T)\times\mathbb{R})}\leq C||h||_{\mathcal{C}_\xF((s,-T_0)\times\mathbb{R})}.$$
It follows that $\psi^s$ can be extended past time $-T,$ unless $T=T_0.$ Moreover, (\ref{estfixmef}) is satisfied as well and $\psi^s$ also satisfies the orthogonality condition.\hfill$\Box$

\bigskip
\noindent{ \bf Proof of Proposition \ref{prop1} }
 Take a sequence $s_j\rightarrow-\infty$ and  $\psi_j=\psi^{s_j}$ where $\psi^{s_j}$ is the function   (\ref{proci}) with $s=s_j.$ Then by (\ref{estfix}), we can find a subsequence $\{\psi_j\}$ and $\psi$ such that $\psi_j\rightarrow\psi$ locally uniformly in $(-\infty,-T_0)\times\mathbb{R}.$

Using (\ref{estfix}) and standard parabolic theory we have that $\psi$ is a solution of (\ref{proci}) and satisfies (\ref{estfix**}).
The proof is concluded.
\section{The nonlinear problem}
Going back to the nonlinear problem, function $\psi$  is a solution of (\ref{mainpro}) if and only if  $\psi\in C_\xF((-\infty,-T_0)\times\mathbb{R})$ solves the fixed point
problem
\be
  \psi= B(\psi ), \label{2.14}
\ee
where
$$
B(\psi ) := A(E(\psi))
$$
and $A$ is the operator in Proposition \ref{fixth}.

Let $T_0>1,$ we define
$$\xL=\{h\in C^1(-\infty,-T_0]:\;\sup_{t\leq -T_0}|h(t)|+\sup_{t\leq -T_0}|t||h'(t)|<1\}$$
and
$$||h||_\xL=\sup_{t\leq -T_0}(|h(t)|)+\sup_{t\leq -T_0}(|t||h'(t)|).$$

The main goal in this section is to prove the following Proposition.
\begin{proposition}
Let $\xs<\sqrt{2}$ and $\xn=\frac{\sqrt{2}-\xs}{2\sqrt{2}}$. There exists number $T_0> 0,$ depending only on $\xs$ such that for any given functions $h$ in $\xL,$ there is a solution  $ \psi= 	 \Psi(h)$ of (\ref{2.14}), with respect $\xi=\xi^0+h.$ The solution $\psi$
satisfies the orthogonality conditions (2.9)-(2.10).
Moreover, the following estimate holds
\be
||\Psi(h_1)-\Psi(h_2)||_{\mathcal{C}_\xF((-\infty,-T_0)\times\mathbb{R})}\leq \frac{C}{T_0^\xn}||h_1-h_2||_\xL,\label{diafora}
\ee
where $C$ is a universal constant.\label{mainproposition}
\end{proposition}

To prove Proposition \ref{mainproposition} we need to prove some lemmas first.

Set
$$X_{T_0}=\{\psi:\;||\psi||_{\mathcal{C}_\xF((-\infty,-T_0)\times\mathbb{R})}<2\frac{C_0}{T_0^\xn}\},$$
for some fixed constant $C_0.$

We denote by $N(\psi,h)$ the function $N(\psi)$ in \eqref{mainlem} with respect $\psi$ and $\xi=\xi^0+h.$
 Also we denote by $z_i$ the respective function in \eqref{z} with respect $\xi=\xi_i=\xi^0+h_i,$ $i=1,2.$
\begin{lemma}
Let $h_1,\;h_2\in \xL$ and $\psi_1,\;\psi_2\in X_{T_0}.$
Then there exists a constant $C=C(C_0)$ such that
\begin{align*}
||N(\psi_1,h_1)&-N(\psi_2,h_2)||_{\mathcal{C}_\xF((-\infty,-T_0)\times\mathbb{R})}\\ &
\leq \frac{C}{T_0^\xn}\left(||\psi_1-\psi_2||_{\mathcal{C}_\xF((-\infty,-T_0)\times\mathbb{R})}+||h_1-h_2||_{\xL}\right).
\end{align*}\label{dia1}
\end{lemma}
\begin{proof}
First we will prove that there exists constant $C>0$ which depends only on $C_0$ such that
\be
||N(\psi_1,h_1)-N(\psi_2,h_1)||_{\mathcal{C}_\xF((-\infty,-T_0)\times\mathbb{R})}\leq \frac{C}{T_0^\xn}||\psi_1-\psi_2||_{\mathcal{C}_\xF((-\infty,-T_0)\times\mathbb{R})}.\label{n1}
\ee
By straightforward calculation we can easily show that
$$|N(\psi_1,h_1)-N(\psi_2,h_1)|\leq \frac{C}{T_0^\xn}|\psi_1-\psi_2|(\xF+\xF^2),$$
where the constant $C$ depend on $C_0$ and the proof of \eqref{n1} follows.

Now we will prove that
\be
||N(\psi_2,h_1)-N(\psi_2,h_2)||_{\mathcal{C}_\xF((-\infty,-T_0)\times\mathbb{R})}\leq C||h_1-h_2||_{\xL}.\label{n2}
\ee
where the constant $C$ depends on $C_0.$

By straightforward calculations we have
\begin{align}\nonumber
|N(\psi_2,h_1)-N(\psi_2,h_2)|&=|-(z_1+\psi_2)^3+z_1^3+3z_1^2\psi_2+(z_2+\psi_2)^3-z_2^3|-3z_2^2\psi_2\\
&\leq \frac{C}{T_0^\xn}|h_1-h_2|\xF^2,\label{n2*}
\end{align}
which implies \eqref{n2}.
By \eqref{n1} and \eqref{n2} the result follows.
\end{proof}
We denote by $E(\psi,h)$ the function $N(\psi)$ in \eqref{mainlem} with respect $\psi$ and $\xi=\xi^0+h.$
\begin{lemma}
Let $h_1,\;h_2\in \xL.$
Then there exists constant $C=C(C_0)$ such that
\be
||E(h_1)-E(h_2)||_{\mathcal{C}_\xF((-\infty,-T_0)\times\mathbb{R})}\leq \frac{C}{T_0^\xn}||h_1-h_2||_{\xL}.
\ee\label{dia2}
\end{lemma}
\begin{proof}
Set $\xi=\xi^0+h_1,$ $\xz=\xi^0+h_2.$
Let $$\frac{\xi_j^0(t)+\xi_{j-1}^0(t)}{2}\leq x\leq \frac{\xi_j^0(t)+\xi_{j+1}^0(t)}{2},\;j=1,...,k,$$ with $\xi_0^0=-\infty$ and $\xi_{k+1}^0=\infty.$
Note here that, there exists $\xm\in[-1,1]$ such that
\begin{align*}
|w(x-\xi_{j-1}(t))-w(x-\xz_{j-1}(t))|&\leq C|h_1-h_2||w'(x-\xi_{j-1}^0(t)+\xm)|\\ \nonumber
&\leq C|h_1-h_2||w'(x-\xi_{j-1}^0(t))|.
\end{align*}
Thus in view of the proof of Lemma \ref{remark} and the above inequality we have
\begin{align*}
|f(z_1(t,x))&-\sum_{j=1}^{k}(-1)^{j+1}f(w(x-\xi_j))-f(z_2(t,x))+\sum_{j=1}^{k}(-1)^{j+1}f(w(x-\xz_j))\\
&\leq C|h_1-h_2||w'(x-\xi_{j-1}^0(t))|.
\end{align*}

Also, we can easily show that
$$
|\sum_{j=1}^{k}(-1)^{j+1}w'(x-\xi_j(t))\xi_j'(t)-\sum_{j=1}^{k}(-1)^{j+1}w'(x-\xz_j(t))\xz'(t)|\leq \frac{C}{t}||h_1-h_2||_{\xL}.
$$

But for any $$\frac{\xi_j^0(t)+\xi_{j-1}^0(t)}{2}\leq x\leq \frac{\xi_j^0(t)+\xi_{j+1}^0(t)}{2},\;j=1,...,k,$$
we have
$$\frac{1}{\xF}\leq  C|t|^{\frac{\xs}{\sqrt{2}}}$$
and $$\frac{1}{\xF}|w'(x-\xi_{j-1}^0(t))|\leq C|t|^{-\xn}.$$
Combining all above we have the desired result.
\end{proof}
\begin{lemma}
Let $h_1,\;h_2\in \xL,$ $\psi_1,\;\psi_2,\;\psi\in X.$  Also let $C(\psi,h,t)=(c_1(t),...,c_k(t))$ satisfy

 \begin{align}\nonumber
\sum_{i=1}^kc_i(t)&\int_{\mathbb{R}}w'(x-\xi_i(t))w'(x-\xi_j(t))dx\\ \nonumber
&=\int_{\mathbb{R}}\left(-f'(w(x-\xi_j))+f'(z(t,x))\right)\psi(t,x)w'(x-\xi_j(t)) dx\\ \nonumber
&-\xi'_j(t)\int_{\mathbb{R}}\psi(t,x)w''(x-\xi_j(t))dx\\ \nonumber
&+\int_{\mathbb{R}}(E(h)+N(\psi,h))w'(x-\xi_j(t))dx,\qquad\forall j=1,...,k,\;t<-T.
\end{align}
 with respect $\psi$ and $\xi=\xi^0+h.$
Then
\be
|C(\psi_1,h_1,t)-C(\psi_2,h_2,t)|\leq \frac{C}{|t|^{1+\frac{\xs}{2\sqrt{2}}}}||\psi_1-\psi_2||_{\mathcal{C}_\xF((-\infty,-T_0)\times\mathbb{R})}+\frac{C}{|t|^{\xn+\frac{\xs}{\sqrt{2}}}}||h_1-h_2||_{\xL},
\ee
for some positive constant $C_0$ which depend only on $C_0.$
\label{dia3}
\end{lemma}
\begin{proof}
For the proof of Lemma, we do very similar calculations like in Lemmas \ref{cilemma}, \ref{dia1}, \ref{dia2} and we omit it.
\end{proof}

\medskip
\emph{Proof of Proposition \ref{mainproposition} }
a)
We consider the operator $B: \mathcal{C}_\xF((-\infty,-T_0)\times\mathbb{R})\rightarrow \mathcal{C}_\xF((-\infty,-T_0)\times\mathbb{R}),$
where $B(\psi)$ denotes the solution to (\ref{2.14}).
We will show that the map $B$ defines a contraction mapping and we will apply the fixed point theorem to it.
First we note by Lemma \ref{remark} and Proposition \ref{fixth} that
$$
||B(0)||_{\mathcal{C}_\xF((-\infty,-T_0)\times\mathbb{R})}\leq \frac{C_0}{T_0^\xn}
$$
and by Proposition \ref{fixth} and Lemma \ref{cilemma*}
\bea\nonumber
&&||B(\psi_1)-B(\psi_2)||_{\mathcal{C}_\xF((-\infty,-T_0)\times\mathbb{R})}\\ \nonumber
&\leq& \frac{C}{T_0^\xn}
\left(||\psi_1-\psi_2||_{\mathcal{C}_\xF((-\infty,-T_0)\times\mathbb{R})}\right),
\eea
providing
$$
||\psi_i||_{\mathcal{C}_\xF((-\infty,-T_0)\times\mathbb{R})}\leq2\frac{C_0}{T_0^\xn}.
$$
Thus if we choose $T_0$ big enough we can apply the fix point theorem in
$$X_{T_0}=\{\psi:\;||\psi||_{\mathcal{C}_\xF((-\infty,-T_0)\times\mathbb{R})}<2\frac{C_0}{T_0^\xn}\},$$
to obtain that there exists $\psi$ such that $B(\psi)=\psi.$

b) For simplicity we set  $\psi^1 = 	\Psi(h_1)$ and
 $\psi^2 = 	\Psi(h_2).$ The estimate will be obtained by applying the estimate \eqref{estfix}. However, because each
 $\psi^i$ satisfies the orthogonality conditions (\ref{orthcond})  with $\xi(t) = \xi^i(t) := \xi^0(t)+ h_i(t),$ the
difference  $\psi^1-\psi^2$ doesn't satisfy an exact orthogonality condition. To overcome this technical difficulty
we will consider instead the difference $Y :=  \psi^1-\overline{\psi}^2,$ where

$$\overline{\psi}^2=\psi^2-\sum_{i=1}^k\xl_i(t)w'(x-\xi_i^1),$$
with
$$
\sum_{i=1}^k\xl_i(t)\int_{\mathbb{R}}w'(x-\xi_i^1(t))w'(x-\xi_j^1(t))dx=\int_{\mathbb{R}}\psi^2(t,x)w'(x-\xi_j^1(t)) dx,
$$
$j=1,...,k.$
Clearly, $Y$ satisfies the orthogonality conditions (\ref{orthcond}) with $\xi(t) = \xi^1(t).$ Denote by $L^i_t$ the
operator
$$
L^i_t\psi^i=\psi^i_t-\psi^i_{xx}+f'(z^i(t,x))\psi^i.
$$
By Lemmas \ref{dia1}, \ref{dia2} and \ref{dia3}
and the fact that
$$\frac{w'(x-\xi_i^1)}{\xF}\leq C|t|^{\frac{\xs}{\sqrt{2}}},\quad\forall\;k=1,...k$$
we can easily prove
\begin{align}\nonumber
||Y||_{\mathcal{C}_\xF((-\infty,-T_0)\times\mathbb{R})}&\leq \frac{C}{T_0^{\xn}}\left(||\psi_1-\psi_2||_{\mathcal{C}_\xF((-\infty,-T_0)\times\mathbb{R})}+||h_1-h_2||_{\xL}\right) \\ &+C\left(\sum_{i=1}^k\sup_{t\in(-\infty,-T_0)}|t|^{\frac{\xs}{\sqrt{2}}}|\xl_i(t)|\right).\label{pao0}
\end{align}
Now, by orthogonality conditions \eqref{orthcond} and \eqref{****}, we have
\bea\nonumber
\left|\int_{\mathbb{R}}\psi^2(t,x)w'(x-\xi_j^1(t))dx\right|&=&\left|\int_{\mathbb{R}}\psi^2(t,x)(w'(x-\xi_j^1(t))-w'(x-\xi_j^2(t)))dx\right|\\
&\leq& \frac{C}{T_0^\xn}|t|^{-\frac{\xs}{\sqrt{2}}}||h_1-h_2||_{\xL}.\label{pao1}
\eea

Now
\begin{align}\nonumber
&\left|\frac{d\int_{\mathbb{R}}\psi^2(t,x)w'(x-\xi_j^1(t))dx}{dt}\right|\\
&=\left|\frac{d\int_{\mathbb{R}}\psi^2(t,x)(w'(x-\xi_j^1(t))-w'(x-\xi_j^2(t)))dx}{dt}\right|.\label{pao2}
\end{align}
But
\begin{align*}
&\int_{\mathbb{R}}\psi^2_t(t,x)(w'(x-\xi_j^1(t))-w'(x-\xi_j^2(t)))dx\\
&=-\int_{\mathbb{R}}\psi^2_{xx}(t,x)(w'(x-\xi_j^1(t))-w'(x-\xi_j^2(t)))dx\\
&+\int_{\mathbb{R}}L^2_t\psi^2(w'(x-\xi_j^1(t))-w'(x-\xi_j^2(t)))dx\\
&-\int_{\mathbb{R}}f'(z^2(t,x))\psi^2(t,x)(w'(x-\xi_j^1(t))-w'(x-\xi_j^2(t)))dx\\
&=\int_{\mathbb{R}}\psi^2(t,x)(w'''(x-\xi_j^1(t))-w'''(x-\xi_j^2(t)))dx\\
&+\int_{\mathbb{R}}L^2_t\psi^2(w'(x-\xi_j^1(t))-w'(x-\xi_j^2(t)))dx\\
&-\int_{\mathbb{R}}f'(z^2(t,x))\psi^2(t,x)(w'(x-\xi_j^1(t))-w'(x-\xi_j^2(t)))dx.
\end{align*}
By the fix point argument in a) we have that
\be
\left|\int_{\mathbb{R}}\psi^2_t(t,x)(w'(x-\xi_j^1(t))-w'(x-\xi_j^2(t)))dx\right|\leq \frac{C}{T_0^\xn}|t|^{-\frac{\xs}{\sqrt{2}}}||h_1-h_2||_{\xL}.\label{pao3}
\ee
By \eqref{pao1}, \eqref{pao2}, \eqref{pao3} and definitions of $\xl_i$ we have that
$$
 |\xl_i(t)|+|\xl_i'(t)|\leq \frac{C}{T_0^\xn}|t|^{-\frac{\xs}{\sqrt{2}}}||h_1-h_2||_{\xL}.
$$
Combining all above we have that
$$
||Y||_{\mathcal{C}_\xF^{0}((-\infty,-T_0)\times\mathbb{R})}\leq \frac{C}{T_0^{\xn}}||\psi_1-\psi_2||_{\mathcal{C}_\xF((-\infty,-T_0)\times\mathbb{R})}+C||h_1-h_2||_{\xL}.
$$
But
\begin{align*}
||\psi_1-\psi_2||_{\mathcal{C}_\xF((-\infty,-T_0)\times\mathbb{R})}&\leq ||Y||_{\mathcal{C}_\xF((-\infty,-T_0)\times\mathbb{R})}+C\left(\sum_{i=1}^k\sup_{t\in(-\infty,-T_0)}|t|^{\frac{\xs}{\sqrt{2}}}|\xl_i(t)|\right)\\
&\leq  \frac{C}{T_0^{\xn}}||\psi_1-\psi_2||_{\mathcal{C}_\xF((-\infty,-T_0)\times\mathbb{R})}+\frac{C}{T_0^\xn}||h_1-h_2||_{\xL}
\end{align*}
and the proof of inequality \eqref{diafora} follows if we choose $T_0$ big enough.\hfill$\Box$
\setcounter{equation}{0}
\section{the choice of $\xi_i$}\label{xiint}
Let $T_0$ big enough, $\frac{\sqrt{2}}{2}<\xs<\sqrt{2}$ and $\psi\in\mathcal{C}_\xF((-\infty,-T_0)\times\mathbb{R})$ be the solution of the problem (\ref{mainpro}).
We want to find $\xi_i$ such that
\begin{align}\nonumber
0&=\int_{\mathbb{R}}\left(-f'(w(x-\xi_j(t)))+f'(z(t,x))\right)\psi(t,x)w'(x-\xi_j(t)) dx\\ \nonumber
&-\xi'_j(t)\int_{\mathbb{R}}\psi(t,x)w''(x-\xi_j(t))dx\\ \nonumber
&+\int_{\mathbb{R}}(E+N(\psi))w'(x-\xi_j(t))dx,\qquad\forall j=1,...,k,\;t<-T,
\end{align}
where
$$E=\sum_{j=1}^{k}(-1)^{j+1}w(x-\xi_j(t))\xi_j'+f(z(t,x))-\sum_{j=1}^{k}(-1)^{j+1}f(w(x-\xi_j(t))),$$
$$N(\psi)=f(\psi(t,x)+z(t,x))-f(z(t,x))-f'(z(t,x))\psi.$$

First we study the error term $E.$
Let $1<j<k,$ then we have that
\begin{align}\nonumber
&\int_{\mathbb{R}}\left(f(z(t,x))-\sum_{i=1}^{k}(-1)^{i+1}f(w(x-\xi_i(t)))\right)w'(x-\xi_j(t))dx\\ \nonumber
&=\int_{\mathbb{R}}\left(f(z(t,x+\xi_j(t)))-\sum_{i=1}^{k}(-1)^{i+1}f(w(x+\xi_j(t)-\xi_i(t)))\right)w'(x)dx.
\end{align}
For simplicity we assume that $i$ is even.
Set
\begin{align*}
g&=\sum_{i=1}^{j-2}(-1)^{i+1}\left(w(x+\xi_j(t)-\xi_{i}(t))-1\right)
\\&+\sum_{i=j+2}^{k}(-1)^{i+1}\left(w(x+\xi_j(t)-\xi_{i}(t))+1\right),
\end{align*}
$$g_1=w(x+\xi_j-\xi_{j-1})-1$$
and
$$g_2=w(x+\xi_j-\xi_{j+1})+1.$$

Using the fact that $\int_{\mathbb{R}}f(w(x))w'(x)dx=0,$ we have
\begin{align}\nonumber
&\int_{\mathbb{R}}f(z(t,x-\xi_j(t))w'(x)dx\\ \nonumber
&=\int_{\mathbb{R}}\left(g+g_1-w(x)+g_2)\right)\left(1-\left(g+g_1+g_2-w(x)\right)^2\right)w'(x)dx\\ \nonumber
&=\int_{\mathbb{R}}\left(g_1+g_2-3w^2(x)g_1-3w^2(x)g_2+3w(x)g_1^2+3w(x)g_2^2-g_1^3-g_2^3\right)w'(x)dx\\
&+\int_{\mathbb{R}}F_0(t,x)w'(x)dx,\label{1}
\end{align}
where $$F_0(t,x)=O(g)+O(g_1g_2).$$
We note that
\begin{align}\nonumber
\int_{\mathbb{R}}|g|w'(x)dx\leq &C\sum_{i=1,\;\; i\neq j-1,j,j+1}^{k}e^{-\xs\left|\xi_i(t)-\xi_j(t)\right|},\\ \nonumber
\int_{\mathbb{R}}|g_1g_2|w'(x)dx\leq &Ce^{-\sqrt{2}\left|\xi_{j+1}(t)-\xi_{j-1}(t)\right|}.
\end{align}

Let $F_1(t,x)=\sum_{i=1,\;\; i\neq j-1,j,j+1}^{k}f(w(x+\xi_i(t)-\xi_j(t))),$ then
\begin{align}\nonumber
&\int_{\mathbb{R}}\left(\sum_{i=1}^{k}(-1)^{j+1}f(w(x+\xi_j(t)-\xi_i(t)))\right)w'(x)dx\\ \nonumber &=\int_{\mathbb{R}}\left(f(g_1+1)+f(g_2-1)\right)w'(x)dx
+\int_{\mathbb{R}}F_1(t,x)w'(x)dx\\
&=\int_{\mathbb{R}}(-2g_1-3g_1^2-g_1^3-2g_2+3g_2^2-g_2^3)w'(x)dx+\int_{\mathbb{R}}F_1(t,x)w'(x)dx.\label{2}
\end{align}
Also we have that
$$\int_{\mathbb{R}}|F_1(t,x)|w'(x)dx\leq C\sum_{i=1,\;\; i\neq j-1,j,j+1}^{k}e^{-\xs\left|\xi_i(t)-\xi_j(t)\right|}.$$
By (\ref{1}), (\ref{2}) we have
\begin{align}\nonumber
&\int_{\mathbb{R}}\left(f(z(t,x-\xi_j(t)))-\sum_{i=1}^{k}(-1)^{i+1}f(w(x+\xi_j(t)-\xi_i(t)))\right)w'(x)dx\\ \nonumber
&=3\int_{\mathbb{R}}(g_1+g_2)(1-w^2(x))w'(x)dx+3\int_{\mathbb{R}}g_1^2(1+w(x))w'(x)dx \\ \nonumber
&+3\int_{\mathbb{R}}g_2^2(w(x)-1)w'(x)dx
+\int_{\mathbb{R}}F_0(t,x)w'(x)dx-\int_{\mathbb{R}}F_1(t,x)w'(x)dx
\end{align}
and
\begin{align}\nonumber
&\int_{\mathbb{R}}g_1(1-w^2(x))w'(x)dx\\ \nonumber
&=\int_{\mathbb{R}}\frac{-2e^{-\frac{\sqrt{2}}{2}(x+\xi_j-\xi_{j-1})}}{e^{\frac{\sqrt{2}}{2}(x+\xi_j-\xi_{j-1})}+e^{-\frac{\sqrt{2}}{2}(x+\xi_j-\xi_{j-1})}}
(1-w^2(x))w'(x)dx\\ \nonumber
&=-2e^{-\sqrt{2}(\xi_j-\xi_{j-1})}\int_{\mathbb{R}}\frac{1}{e^{\sqrt{2}x}+e^{-\sqrt{2}(\xi_j-\xi_{j-1})}}
(1-w^2(x))w'(x)dx\\ \nonumber
&=-2e^{-\sqrt{2}(\xi_j-\xi_{j-1})}\int_{\mathbb{R}}e^{-\sqrt{2}x}(1-w^2(x))w'(x)dx\\ \nonumber
&-2e^{-\sqrt{2}(\xi_j-\xi_{j-1})}\int_{\mathbb{R}}\left(\frac{1}{e^{\sqrt{2}x}+e^{-\sqrt{2}(\xi_j-\xi_{j-1})}}-e^{-\sqrt{2}x}\right)(1-w^2(x))w'(x)dx\\ \nonumber
&=-2e^{-\sqrt{2}(\xi_j-\xi_{j-1})}\\ \nonumber
&\times\left(\int_{\mathbb{R}}e^{-\sqrt{2}x}(1-w^2(x))w'(x)dx+\int_{\mathbb{R}}F_2(t,x)(1-w^2(x))w'(x)dx\right).
\end{align}
Now
\begin{align}\nonumber
&\left|\int_{\mathbb{R}}F_2(t,x)(1-w^2(x))w'(x)dx\right|\\ \nonumber
&=e^{-\sqrt{2}(\xi_j-\xi_{j-1})}
\int_{\mathbb{R}}\frac{1}{e^{\sqrt{2}x}\left(e^{\sqrt{2}x}+e^{-\sqrt{2}(\xi_j-\xi_{j-1})}\right)}(1-w^2(x))w'(x)dx\\ \nonumber
&\leq C (\xi_j-\xi_{j-1})e^{-\sqrt{2}(\xi_j-\xi_{j-1})}.
\end{align}

Similarly for $g_2$ we have
\begin{align}\nonumber
&\int_{\mathbb{R}}g_2(1-w^2(x))w'(x)dx=2e^{-\sqrt{2}(\xi_j-\xi_{j-1})}\\ \nonumber
&\times\left(\int_{\mathbb{R}}e^{-\sqrt{2}x}(1-w^2(x))w'(x)dx+\int_{\mathbb{R}}F_3(t,x)(1-w^2(x))w'(x)dx\right),
\end{align}
where
\begin{align*}
&\left|\int_{\mathbb{R}}F_3(t,x)(1-w^2(x))w'(x)dx\right|\\
&\qquad\leq C (\xi_{j+1}-\xi_{j})e^{-\sqrt{2}(\xi_{j+1}-\xi_{j})}.
\end{align*}
Now
\begin{align*}
\int_{\mathbb{R}}&g_1^2(1+w(x))w'(x)dx\\
&\leq Ce^{-2\sqrt{2}(\xi_j-\xi_{j-1})}
\int_{\mathbb{R}}\frac{1}{e^{2\sqrt{2}x}+e^{-2\sqrt{2}(\xi_j-\xi_{j-1})}}(1+w(x))w'(x)dx.
\end{align*}
But
\begin{align}\nonumber
&\int_{\mathbb{R}}\frac{1}{e^{2\sqrt{2}x}+e^{-2\sqrt{2}(\xi_j-\xi_{j-1})}}(1+w(x))w'(x)dx\\ \nonumber
&=\int_{-\infty}^{-\xi_j-\xi_{j-1}}\frac{1}{e^{2\sqrt{2}x}+e^{-2\sqrt{2}(\xi_j-\xi_{j-1})}}(1+w(x))w'(x)dx \\ \nonumber
&+\int^{0}_{-\xi_j-\xi_{j-1}}\frac{1}{e^{2\sqrt{2}x}+e^{-2\sqrt{2}(\xi_j-\xi_{j-1})}}(1+w(x))w'(x)dx\\ \nonumber
&+\int^{\infty}_{0}\frac{1}{e^{2\sqrt{2}x}+e^{-2\sqrt{2}(\xi_j-\xi_{j-1})}}(1+w(x))w'(x)dx\\ \nonumber
&\leq C((\xi_j-\xi_{j-1})+1).
\end{align}
Thus we have
$$\int_{\mathbb{R}}g_1^2(1+w(x))w'(x)dx\leq C(\xi_j-\xi_{j-1})e^{-2\sqrt{2}(\xi_j-\xi_{j-1})}.$$
Similarly
$$\int_{\mathbb{R}}g_2^2(1-w(x))w'(x)dx\leq C(\xi_{j+1}-\xi_{j})e^{-2\sqrt{2}(\xi_{j+1}-\xi_{j})}.$$
By assumptions on $\psi$ we have
\begin{align}\nonumber
&\int_{\mathbb{R}}|N(\psi)|w'(x-\xi_j(t))dx\\ \nonumber
&\leq C\int_{\mathbb{R}}\xF^2(t,x)w'(x-\xi_j(t))dx=C\int_{\mathbb{R}}\xF^2(t,x+\xi_j(t))w'(x)dx\\ \nonumber
&\leq C\sum_{i=1}^{k}\int_{\frac{\xi_i^0(t)+\xi_{i-1}^0(t)}{2}-\xi_j(t)}^{\frac{\xi_i^0(t)
+\xi_{i+1}^0(t)}{2}-\xi_j(t)}\left(e^{2\xs(-x-\xi_j(t)+\xi_{i-1}^0)}+e^{2\xs(x+\xi_j(t)-\xi_{i+1}^0)}\right)w'(x)dx.
\end{align}
Now note that
\begin{align}\nonumber
\int_{\frac{\xi_j^0(t)+\xi_{j-1}^0(t)}{2}-\xi_j(t)}^{\frac{\xi_j^0(t)
+\xi_{j+1}^0(t)}{2}-\xi_j(t)}e^{2\xs(x+\xi_j(t)-\xi_{j+1}^0)}w'(x)dx\leq Ce^{(-\xs-\frac{\sqrt{2}}{2})(\xi_{j+1}-\xi_j(t))}.
\end{align}
Thus we can easily prove that
$$\int_{\mathbb{R}}|N(\psi)|w'(x-\xi_j(t))dx\leq C\sum_{i=1,\;\; i\neq j}^{k}e^{(-\xs-\sqrt{2})\left|\xi_i(t)-\xi_j(t)\right|}.$$

Also we have
$$\sum_{i=1}^{k}(-1)^{j+1}\xi_i'\int_{\mathbb{R}}w'(x-\xi_i(t))w'(x-\xi_j(t))dx=-\xi'_j(t)\int_{\mathbb{R}}|w'(x)|^2dx+F_4(t),$$
where
$$|F_4(t)|\leq C\sum_{i=1,\;i\neq j}^{k}|\xi_i'|e^{-\xs|\xi_i-\xi_j|}.$$

Finally
\begin{align}\nonumber
&\left|\xi'_j(t)\int_{\mathbb{R}}\psi(t,x)w''(x-\xi_j(t))dx\right|\leq C|\xi'_j(t)|\int_{\mathbb{R}}\xF(t,x+\xi_j(t))w''(x)dx\\ \nonumber
&\qquad\leq C|\xi'_j(t)|\left(e^{(-\frac{\xs}{2}-\sqrt{2})\left|\xi_{j+1}(t)-\xi_j(t)\right|}+e^{(-\frac{\xs}{2}-\sqrt{2})\left|\xi_{j-1}(t)-\xi_j(t)\right|}\right).
\end{align}

Similarly for   $j=1,...,k,$ we can reach at the respective ODE, for $\xi=(\xi_1,...,\xi_k)$
\be
\frac{1}{\xb}\xi_j'-e^{-\sqrt{2}(\xi_{j+1}-\xi_j)}+e^{-\sqrt{2}(\xi_{j}-\xi_{j-1})}=F_i(\xi',\xi),\qquad j=1,2,...,k,\;\;t\in(0,-T_0],\label{ode*}
\ee
with $\xi_{k+1}=\infty$ and $\xi_0=-\infty.$

 We recall here that, we assume $T_0>1$ and we denote by
$$\xL=\{h\in C^1(-\infty,-T_0]:\;\sup_{t\leq -T_0}|h(t)|+\sup_{t\leq -T_0}|t||h'(t)|<1\}$$
and
$$||h||_\xL=\sup_{t\leq -T_0}(|h(t)|)+\sup_{t\leq -T_0}(|t||h'(t)|).$$
We set
$$\overline{F}(h',h)=F(\xi',\xi),$$
where $\xi=\xi^0+h.$
Working like above and Lemmas \ref{pao1}, \ref{pao2}, \ref{pao3} and using \eqref{diafora} we have the following result.

\begin{proposition}
Let $\frac{\sqrt{2}}{2}<\xs<\sqrt{2}$ and $h,\;h_1,\;h_2\in \xL.$ Then there exists a constant $C=C(\xs)$ such that
$$|\overline{F}(h',h)|\leq \frac{C}{|t|^{\frac{1}{2}+\frac{\xs}{\sqrt{2}}}},$$
and
$$|\overline{F}(h'_1,h_2)-\overline{F}(h'_1,h_2)|\leq \frac{C}{|t|^{\frac{1}{2}+\frac{\xs}{\sqrt{2}}}}||h_1-h_2||_{\xL}.$$\label{remark2}
\end{proposition}

In the rest of this section we will study the system \ref{ode*} using some ideas from \cite{dds}.
\subsection{the choice of $\xi^0$}\label{ksi0}
Let $k\geq4$ be an even number. First, we want to find a solution of the problem
\be
\frac{1}{\xb}\xi_j'-e^{-\sqrt{2}(\xi_{j+1}-\xi_j)}+e^{-\sqrt{2}(\xi_{j}-\xi_{j-1})}=0,\qquad j=1,2,...,k,\;\;t\in(0,-T_0],\label{ode}
\ee
with $\xi_{k+1}=\infty$ and $\xi_0=-\infty.$
We set
$$R_l(\xi):=-e^{-\sqrt{2}(\xi_{j+1}-\xi_j)}+e^{-\sqrt{2}(\xi_{j}-\xi_{j-1})}$$
and
$$
\textbf{R}(\xi)=\left[ \begin{array}{ccc}
R_1(\xi) \\
\vdots  \\
R_k(\xi)  \end{array} \right]
.$$
We want to solve the system $\xi'+\xb\textbf{R}(\xi)=0.$ To do so we find first a convenient representation of the operator $\textbf{R}(\xi).$ Let us consider the auxiliary variables

$$
\textbf{v}:=\left[ \begin{array}{ccc}
\mathbf{\overline{v}} \\
v_k
  \end{array} \right],
  \qquad
\mathbf{\overline{v}}= \left[ \begin{array}{ccc}
v_1 \\
\vdots  \\
v_{k-1}  \end{array} \right],
$$
defined in terms of $\xi$ as
$$v_l=\xi_{l+1}-\xi_l\;\;\;\mathrm{with}\;l=1,...,k-1,\qquad v_k=\sum_{l=1}^k\xi_l$$
and define the operators
$$
\mathbf{S}(\textbf{v}):=\left[ \begin{array}{ccc}
\overline{\mathbf{S}}(\mathbf{\overline{v}}) \\
0
  \end{array} \right],
  \qquad
\overline{\mathbf{S}}(\mathbf{\overline{v}})= \left[ \begin{array}{ccc}
S_1(\overline{\mathbf{v}}_1) \\
\vdots  \\
S_{k-1}(\overline{\mathbf{v}}_1)  \end{array} \right],
$$
where $S_l(\overline{\mathbf{v}}):R_{l+1}(\xi)-R_l(\xi)=$
$$
\Bigg\{ \begin{array}{ccc}
2e^{-\sqrt{2}v_1}-e^{\sqrt{2}v_2} &\mathrm{if}\qquad l=1  \\
-e^{\sqrt{2}v_{l-1}}+2e^{-\sqrt{2}v_l}-e^{\sqrt{2}v_{l-1}}&\mathrm{if}\qquad 2\leq l\leq k-2  \\
2e^{-\sqrt{2}v_k}-e^{\sqrt{2}v_{k-1}} &\mathrm{if}\qquad l=k-1
  \end{array}.
$$
Then the operators  $\mathbf{R}$ and $\mathbf{S}$ are in correspondence through the formula
$$\mathbf{S}(\mathbf{v})=\mathbf{B}\mathbf{R}(\mathbf{B}^{-1}\mathbf{v}),$$
where $\mathbf{B}$ is the constant, invertible $k\times k$ matrix
$$\mathbf{\mathbf{B}}=\left[ \begin{array}{ccccc}
-1 & 1 & 0&\cdots&0 \\
0 & -1 & 1&\cdots&0 \\
\vdots & \ddots & \ddots&\ddots&\vdots\\
0&\cdots&0&-1&1\\
1&\ldots& 1&1&1
\end{array} \right]$$
and then through the relation $\xi=\mathbf{B}^{-1}\mathbf{v}$ the system $\xi'+\xb\mathbf{R}(\xi)=0$ is equivalent to $\mathbf{v}'+\xb\mathbf{S}(\mathbf{v})=0,$ which decouples into

\bea
\nonumber
\overline{v}+\xb\overline{\mathbf{S}}(\mathbf{\overline{v}})&=&0,\\ \nonumber
v_k'&=&0,
\eea
where
\be
\overline{\mathbf{S}}(\mathbf{\overline{v}})= \mathbf{C}\left[ \begin{array}{ccc}
e^{-\sqrt{2} v_1} \\
\vdots  \\
e^{-\sqrt{2}v_{k-1}}  \end{array} \right]
,\qquad
\mathbf{C}=\left[ \begin{array}{ccccc}
2 & -1 & 0&\cdots&0 \\
-1 & 2 & -1&\cdots&0 \\
\vdots & \ddots & \ddots&\ddots&\vdots\\
0&\cdots&-1&2&-1\\
0&\ldots& &-1&2
\end{array} \right].\label{C}
\ee
We choose simply $v_k=0$ and loo for a solution $\mathbf{v}^0(t)=(\overline{\mathbf{v}}(t)^0,0)$ of the system, where $\overline{\mathbf{v}}^0(t)$ has the form
\be
\overline{v}_l^0(t)=\frac{1}{\sqrt{2}}\log(-2\sqrt{2}\xb t)+b_l,\label{v}
\ee
for constants $b_l$ to be determined.

Substituting this expression into the system we find the following equations for the numbers $b_l$
$$
\mathbf{C}\left[ \begin{array}{ccc}
e^{-\sqrt{2} b_1} \\
\vdots  \\
e^{-\sqrt{2}b_{k-1}}  \end{array} \right]=\frac{1}{\xb}\left[ \begin{array}{ccc}
1 \\
\vdots  \\
1  \end{array} \right]
.$$
We compute explicitly,
$$b_l=-\frac{1}{\sqrt{2}}\log\left(\frac{1}{2\xb}(k-l)l\right),\qquad l=1,...,k-1.$$
Now we note that
$b_l=b_{k-l}$ for $l=1,..,k-1,$ thus by (\ref{ode}) we have that $$\xi_{k-j+1}=-\xi_{j},\;\;j\leq\frac{k}{2},$$
and
$$\xi_j=\frac{1}{\sqrt{2}}\left(j-\frac{k+1}{2}\right)\log(-2\sqrt{2}\xb t)+\xg_j,$$
where
$$-\xg_j=\xg_{k-j+1}=\frac{1}{2}\sum_{i=j}^{k-j}b_i,\qquad\mathrm{for}\;j\leq\frac{k}{2}.$$

\subsection{the solution of the problem \eqref{ode*}}\label{general}
We keep the notations of the previous subsection, and we write problem \eqref{ode*} in the form
we consider the problem
$$\xi'+\xb\mathbf{R}(\xi)= \mathbf{F}(\xi',\xi),\qquad \mathrm{in} \;(-\infty,- T_0].$$
Let $\xi^0=(\xi_1^0,...,\xi_k^0)^T$ where
$$\xi_j^0=\frac{1}{\sqrt{2}}\left(j-\frac{k+1}{2}\right)\log(-2\sqrt{2}\xb t)+\xg_j,$$
We look for solution of the form $\xi=\xi^0+h.$ Thus $h$ satisfies
\bea\nonumber
h'+\xb D_\xi\mathbf{R}(\xi^0)h&=& \mathbf{F}({\xi^0}'+h',\xi^0+h)+\xb D_\xi\mathbf{R}(\xi^0)h-\xb\mathbf{R}(\xi^0)\\ \nonumber
&=&\mathbf{E}(h',h),\qquad \mathrm{in} \;(-\infty,-T_0].\label{ode1}
\eea
By Proposition \ref{remark2},  we have
\begin{align}\nonumber
|\mathbf{E}(0,0)|&\leq C\left(\frac{1}{|t|}\right)^{\frac{1}{2}+\frac{\xs}{\sqrt{2}}},\\
|\mathbf{E}(h_1',h_1)-\mathbf{E}(h_1',h_1)|&\leq C\left(\frac{1}{t}\right)^{\frac{1}{2}+\frac{\xs}{\sqrt{2}}}|h_1-h_2|+C\left(\frac{1}{t}\right)^{\frac{1}{2}+\frac{\xs}{\sqrt{2}}}|h'_1-h_2'|.\label{estimatesode}
\end{align}
Also we are restricting ourselves to symmetric $\xi,$ then $h$ satisfies the symmetry condition
$$h_{k-j+1}=-h_j,\qquad j\leq\frac{k}{2}.$$
In addition this implies that the solution $\psi$ is even with respect $x$ and thus we have that
\be
E_{k-j+1}=E_j,\qquad j\leq\frac{k}{2}.\label{E1}
\ee

Set
$$v^0=\mathbf{B}\xi^0\qquad\mathrm{and}\qquad p=\mathbf{B}h.$$
Then we have that $\mathbf{E}(h',h)=\mathbf{E}(\mathbf{B}^{-1}h',\mathbf{B}^{-1}h)=\mathbf{E}(p',p),$ and by
$$\mathbf{S}(\mathbf{v})=\mathbf{B}\mathbf{R}(\mathbf{B}^{-1}\mathbf{v}),$$
we have that
$$\mathbf{S}(\mathbf{v^0})=\mathbf{B}\mathbf{R}(\xi^0)\mathbf{B}^{-1}.$$
Thus (\ref{ode1}) is equivalent to
\be
p'+\xb D_v\mathbf{S}(\mathbf{v^0})p=\mathbf{B}\mathbf{E}(p',p):=\mathbf{L}(p',p),\;\;\;\mathrm{in} \;(-\infty,-T_0].\label{ode2}
\ee
By (\ref{E1}) we have that $\mathbf{L}_k=0,$ thus writting $p=(\overline{p},p_k)$ and $\mathbf{L}=(\overline{\mathbf{L}},L_k),$ the latter system decouples as
\bea\nonumber
\overline{p}'+\xb D_{\overline{v}}\mathbf{\overline{S}}(\mathbf{\overline{v}^0})&=&\overline{\mathbf{L}}(\overline{p}',\overline{p}),\;\;\;\mathrm{in} \;(-\infty,-T_0],\\
p_k'&=&0,
\label{ode3}
\eea
where we have simply choose $p_k=0.$

Now, by (\ref{v}) we have
\bea\nonumber
D_{\overline{v}}\mathbf{\overline{S}}(\mathbf{\overline{v}^0})&=&-\sqrt{2}\mathbf{C} \left[ \begin{array}{ccccc}
e^{-\sqrt{2} v_1} & 0 &\cdots &0 \\
0 &e^{-\sqrt{2} v_2} & \cdots &0  \\
\vdots& &\ddots &\vdots  \\
0 &0 &\cdots & e^{-\sqrt{2}v_{k-1}}  \end{array} \right]\\ \nonumber
&=&\frac{1}{2\xb t}\mathbf{C}\left[ \begin{array}{ccccc}
a_1 & 0 &\cdots &0 \\
0 &a_2 & \cdots &0  \\
\vdots& &\ddots &\vdots  \\
0 &0 &\cdots & a_{k-1}  \end{array} \right],
\eea
where $a_l=\frac{1}{2\xb}(k-l)l,\;l=1,...,k-1,$
where the matrix $\mathbf{C}$ is given in (\ref{C}). $\mathbf{C}$ is symmetric and positive definite. Indeed, a straightforward computation yields that its eigenvalues are explicitly given by
$$1,\frac{1}{2},...,\frac{k-1}{k}.$$
We consider the symmetric, positive definite square root matrix of $\mathbf{C}$ and denote it by $\mathbf{C}^{\frac{1}{2}}.$ Then setting
$$\overline{p}=\mathbf{C}^\frac{1}{2}w,\qquad Q(w',w)=\mathbf{C}^{-\frac{1}{2}}\overline{\mathbf{L}}(\mathbf{C}^{\frac{1}{2}}w',\mathbf{C}^{\frac{1}{2}}w),$$
we see that equation (\ref{ode3}) becomes
\be
w'+\frac{1}{2t}\mathbf{A}w=Q(w',w),\label{ode4}
\ee
where
$$\mathbf{A}=\mathbf{C}^\frac{1}{2}\left[ \begin{array}{ccccc}
a_1 & 0 &\cdots &0 \\
0 &a_2 & \cdots &0  \\
\vdots& &\ddots &\vdots  \\
0 &0 &\cdots & a_{k-1}  \end{array} \right]\mathbf{C}^\frac{1}{2}.$$
In particular $\mathbf{A}$ has positive eigenvalues $\xl_1,\xl_2,...,\xl_{k-1}.$ Let the orthogonal matrix $\mathbf{\xL}$ such that $\mathbf{D} = \mathbf{\xL}^T \mathbf{A}\mathbf{\xL},$ where $\mathbf{D}$  is the diagonal matrix such that $A_{ii}=\xl_i,\;i=1,...,k-1.$
Set now
$$\xo=\mathbf{\xL}^Tw,\qquad\mathbf{\xG}(\xo',\xo)=\mathbf{\xL}^TQ(\mathbf{\xL}\xo',\mathbf{\xL}\xo),$$
we have that (\ref{ode4}) becomes equivalent to
\be
\xo'+\frac{1}{2t}\mathbf{D}\xo=\mathbf{\xG}(\xo',\xo),\quad\mathrm{in} (-\infty,-T_0).\label{odepro}
\ee
We will solve (\ref{odepro}) by using the fix point Theorem in a suitable space with initial data $w(T_0)=0$. If $\xo$ is a solution of the problem (\ref{odepro}) with initial data then has the form
\be
\xo_i(t)=-\frac{1}{(-t)^{\sqrt{\xl_i}}}\int^{-t_0}_t (-s)^{\sqrt{\xl_i}}\xG_i(\xo',\xo)ds.\label{odefix}
\ee
Let $A(\xo)$ be a solution of (\ref{odefix}), then $\mathbf{\xG}$ satisfies the same estimates in (\ref{estimatesode}) and we have
\bea
|A(0)|\leq C_1\left(\frac{1}{T_0}\right)^{\frac{\xs}{\sqrt{2}}-\frac{1}{2}}.\label{fragma}
\eea
Similarly
\be
|t||A(0)'|\leq C_2\left(\frac{1}{T_0}\right)^{\frac{\xs}{\sqrt{2}}-\frac{1}{2}},\label{fragma2}
\ee
if we choose $t_0>1.$
Thus we consider the space
$$X=\{h\in C^1(-\infty,-t_0]:\;||h||_\xL\leq2c_0\},$$
where $c_0=C_1+C_2$ the constants in (\ref{fragma}) and (\ref{fragma2}).
Thus
$$
|A(h_1)-A(h_2)|\leq C\left(\frac{1}{T_0}\right)^{\frac{\xs}{\sqrt{2}}-\frac{1}{2}}||h_1-h_2||_\xL,
$$
$$
|t||A'(h_1)-A'(h_2)|\leq
C||h_1-h_2||_\xL\left(\frac{1}{T_0}\right)^{\frac{\xs}{\sqrt{2}}-\frac{1}{2}}.
$$
Thus we have
$$||A(h_1)-A(h_2)||_{\xL}\leq C(\xs)\left(\frac{1}{T_0}\right)^{\frac{\xs}{\sqrt{2}}-\frac{1}{2}}.$$
The result follows by fixed point theorem if we choose $T_0$ big enough.

\bigskip
\noindent\emph{Acknowledgment } This work has been supported by  Fondecyt grants  3140567 and 1150066, Fondo Basal CMM and by Millenium Nucleus CAPDE NC130017.

\end{document}